\documentclass[11pt]{article} 

\usepackage{amssymb}
\usepackage{amsmath,amsthm,amsfonts,epsfig,setspace}

\usepackage[utf8]{inputenc}
\usepackage{cases}

\usepackage{graphicx} 
\graphicspath{ {MATLAB/} }
\usepackage{subcaption}
\usepackage{color}
\usepackage{hyperref}
\usepackage{tikz}
\usepackage{wrapfig}
\usetikzlibrary{intersections,decorations.pathreplacing,decorations.markings,calc,positioning}

\numberwithin{equation}{section}

\usepackage[a4paper,bottom=4.5cm]{geometry}

\newcommand{\nm}{\noalign{\smallskip}}
\newcommand{\beq}{\begin{equation}}
\newcommand{\eeq}{\end{equation}}
\newcommand{\eqnref}[1]{(\ref {#1})}
\newcommand{\ds}{\displaystyle}

\newcommand{\p}{\partial}
\newcommand{\pd}[2]{\frac {\p #1}{\p #2}}
\newcommand{\bigslant}[2]{{\raisebox{.2em}{$#1$}\left/\raisebox{-.2em}{$#2$}\right.}}
\newcommand{\ie}{\textit{i.e.}}

\newtheorem{theorem}{Theorem}[section]

\newtheorem{lemma}[theorem]{Lemma}

\newtheorem{remark}{Remark}
\newtheorem{asump}{Assumption}[section]

\newcommand{\R}{\mathbb{R}}

\newcommand{\Bx}{{x}}
\newcommand{\By}{{y}}

\newcommand{\Scal}{\mathcal{S}}
\newcommand{\Kcal}{\mathcal{K}}
\newcommand{\Acal}{\mathcal{A}}
\newcommand{\dx}{\: \mathrm{d}}
\renewcommand{\O}{\mathcal{O}}
\newcommand{\iu}{\mathrm{i}\mkern1mu}

\title{A high-frequency homogenization approach near the Dirac points in bubbly honeycomb crystals}

\author{
Habib Ammari\thanks{\footnotesize Department of Mathematics, 
ETH Z\"urich, 
R\"amistrasse 101, CH-8092 Z\"urich, Switzerland (habib.ammari@math.ethz.ch, erik.orvehed.hiltunen@sam.math.ethz.ch, sanghyeon.yu@sam.math.ethz.ch).}\and Erik Orvehed Hiltunen\footnotemark[1]  \and Sanghyeon Yu\footnotemark[1]}

\date{} 
  
\begin{document}
\maketitle
\begin{abstract}
In [H. Ammari et al., Honeycomb-lattice Minnaert bubbles. arXiv:1811.03905], the existence of a Dirac dispersion cone in a bubbly honeycomb phononic crystal is shown. The aim of this paper is to prove that, near the Dirac points, the Bloch eigenfunctions is the sum of two eigenmodes. Each eigenmode can be decomposed into two components: one which is slowly varying and satisfies a homogenized equation, while the other is periodic across each elementary crystal cell and is highly oscillating. The slowly oscillating components of the eigenmodes satisfy a system of Dirac equations. 
Our results in this paper proves for the first time a near-zero effective refractive index near the Dirac points for the plane-wave envelopes of the Bloch eigenfunctions in a sub-wavelength metamaterial. They are illustrated by a variety of numerical examples. 
We also compare and contrast the behaviour of the Bloch eigenfunctions in the honeycomb crystal with that of their counterparts in a bubbly square crystal, near the corner of the Brillouin zone, where the maximum of the first Bloch eigenvalue is attained.
\end{abstract}

\def\keywords2{\vspace{.5em}{\textbf{  Mathematics Subject Classification
(MSC2000).}~\,\relax}}
\def\endkeywords2{\par}
\keywords2{35R30, 35C20.}

\def\keywords{\vspace{.5em}{\textbf{ Keywords.}~\,\relax}}
\def\endkeywords{\par}
\keywords{Honeycomb lattice, Dirac cone, Dirac equation, bubble, Minnaert resonance, sub-wavelength bandgap, near-zero effective index.}

\section{Introduction}
Metamaterials are a novel group of materials designed to have special wave characteristics such as bandgaps, negative refractive indices, or sub-wavelength scale resolution in imaging. There have also been demonstrations of materials
with near-zero refractive indices. These materials have a wide number of applications, including low-loss bending transmission, invisibility cloaking, and zero phase-shift propagation \cite{alu, Dubois2017, Hyun2018, NZ2012,Z2013}.

The first near-zero refractive index phononic crystal was theoretically demonstrated in  \cite{NZ2012}, where the effective mass density and reciprocal bulk modulus were shown to vanish simultaneously.  This near-zero effective refractive index is a consequence of the existence of a Dirac dispersion cone in the dispersion relation of the material. The double-zero property is possible because the Dirac cone is located at the centre $\Gamma$ of the Brillouin zone. Single-zero properties have been studied for other locations of the Dirac cone; however, these materials exhibit a low transmittance making them less desirable for applications \cite{Dubois2017,add2,add3}. At high frequencies, the Bloch eigenmodes will oscillate on the microscale of the metamaterial, suggesting that a homogeneous description of the material is overly simplified. Nevertheless, as will be shown in this paper, an effective equation for the envelopes of these Bloch eigenmodes can be derived. 

Metamaterials with Dirac singularities have been experimentally and numerically studied in \cite{torrent,Dirac1,Dirac2}. Proofs of the existence  of a Dirac cone
at  the symmetry point $K$ in honeycomb lattice structures and mathematical analyses of their properties are provided in \cite{honeycomb, jmp, drouot, Dirac6, Dirac2018, tightbind2,tightbind1}. In \cite{Dubois2017}, time-dependent material parameters are used to move the Dirac cone from the point $K$ to the centre $\Gamma$ of the Brillouin zone, enabling a double-zero refractive index.

Sub-wavelength resonators are the building blocks of metamaterials. In acoustics, a gas bubble in a liquid is known to have a resonance frequency corresponding to wavelengths which are several orders of magnitude larger than
the bubble \cite{H3a, Minnaert1933}. This opens up the possibility of creating small-scaled acoustic metamaterials known as \emph{sub-wavelength} metamaterials, whereby the operating frequency corresponds to wavelengths much larger than the device size. The simplicity of the gas bubble makes bubbly media an ideal model for sub-wavelength metamaterials. Many experimentally observed phenomena in bubbly media  \cite{rev1,leroy2, fink,  leroy1, leroy3, rev2, rev3} have been rigorously explained in \cite{Ammari_David,AFLYZ, defect, honeycomb, AZ_hom}. In particular, in \cite{honeycomb},  a bubbly honeycomb crystal is considered, and a Dirac dispersion cone centred at the symmetry point $K$ in the Brillouin zone is shown to exist.

In this paper, we prove the near-zero effective index property around the point $K$ in a bubbly honeycomb crystal at the deep sub-wavelength scale. We will develop a homogenization theory that captures both the macroscopic behaviour of the eigenmodes and the oscillations in the microscopic scale, and demonstrate that the near-zero property holds in the macroscopic scale.

In the homogenization theory of metamaterials, the goal is to map the metamaterial to a homogeneous material with some effective parameters. It has previously been demonstrated that this approach does not apply in the case of  bubbly crystals at ``high'' frequencies, \ie{} away from the centre $\Gamma$ of the Brillouin zone. In \cite{homogenization}, it is shown that around the symmetry point $M$ in the Brillouin zone of a bubbly crystal with a square lattice, the Bloch eigenmodes display oscillatory behaviour on two distinct scales: small scale
oscillations on the order of the size of individual
bubbles, while simultaneously the plane-wave envelope oscillates at a much larger scale and satisfies a homogenized equation. Analogously, we expect the standard homogenization approach to fail for the honeycomb crystal and seek instead a homogenized equation for the envelopes of the eigenmodes. We will demonstrate that this is a near-zero refractive index homogenized equation near the Dirac points. Moreover, we will compare our results with the case of a square lattice crystal, which does not have a linear dispersion relation around the symmetry points of the Brillouin zone, and consequently cannot have effective near-zero refractive index.

This paper is organized as follows. In Section \ref{sec:setup}, we present the eigenvalue problem of the bubbly honeycomb crystal, and state the main results of \cite{honeycomb}. In Section \ref{sec:efunc},  we use layer-potential techniques to compute the Bloch eigenfunctions close to the point $K$ in the asymptotic limit of high density contrast. In Section \ref{sec:hom}, we decompose the Bloch eigenfunctions as the sum of two eigenmodes, each with a slowly oscillating plane-wave envelope. We also derive a system of Dirac equations satisfied by the slowly oscillating components of the eigenmodes  in the vicinity of the Dirac points. This generalizes the result obtained in \cite{Dirac8}, for the Schr\"odinger equation, to wave propagation in sub-wavelength resonant structures. The main result is stated in Theorem \ref{thm:main}, where the Bloch eigenmodes are shown to exhibit the two-scale behaviour as described above. In Section \ref{sec:num}, we numerically illustrate Theorem \ref{thm:main}. 
We show that the macroscopic plane-wave envelope in the honeycomb crystal has a lower order of oscillations compared to the Bloch eigenfunction of the square crystal, and demonstrate the near-zero effective refractive index of the honeycomb crystal. Finally, in Section \ref{sec:conclusion}, we summarise the main results of this paper and briefly discuss the remaining challenges in the field.

%%%%%%%%%%%%%%%%%%%%%%%%%%%%%%%%%%%%%%%%%%%%%%%%%%%%%%%%%%%%%%%%%%%%%%%%%%%%%%%%%%%%%
\section{Problem statement and preliminaries} \label{sec:setup}
In this section, we describe the honeycomb lattice and state the main results of \cite{honeycomb}. 
\subsection{Problem formulation}
We consider a two-dimensional infinite honeycomb crystal in two dimensions depicted in Figure \ref{fig:honeycomb}. Define the lattice $\Lambda$ generated by the lattice vectors
$$ l_1 = L\left( \frac{\sqrt{3}}{2}, \frac{1}{2} \right),~~l_2 = L\left( \frac{\sqrt{3}}{2}, -\frac{1}{2}\right),$$
where $L$ is the lattice constant. Denote by $Y$ a fundamental domain of the given lattice. Here, we take 
$$ Y:= \left\{ s l_1+ t l_2 ~|~ 0 \le s,t \le 1 \right\}. $$
Define the three points $x_0, x_1,$ and $x_2$ as
$$x_0 = \frac{l_1 + l_2}{2}, \quad x_1 = \frac{l_1+l_2}{3}, \quad x_2 = \frac{2(l_1 + l_2)}{3} .$$

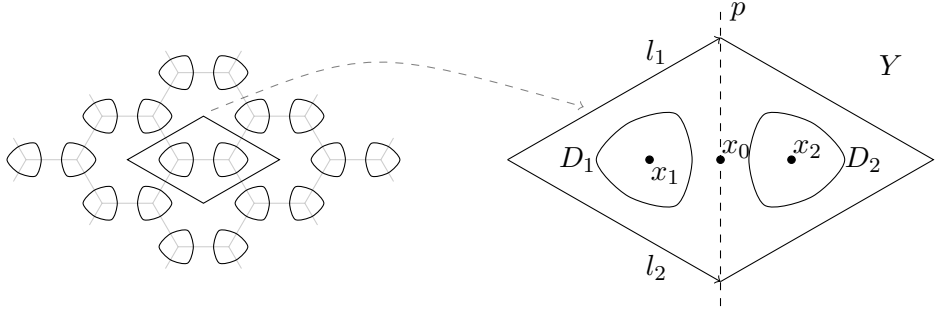
\begin{figure}[tb]
	\centering
	\begin{tikzpicture}
	\begin{scope}[xshift=-5cm,scale=1]
	\coordinate (a) at (1,{1/sqrt(3)});		
	\coordinate (b) at (1,{-1/sqrt(3)});	
	\pgfmathsetmacro{\rb}{0.25pt}
	\pgfmathsetmacro{\rs}{0.2pt}
	
	\draw (0,0) -- (1,{1/sqrt(3)}) -- (2,0) -- (1,{-1/sqrt(3)}) -- cycle; 
	\begin{scope}[xshift = 1.33333cm]
	\draw plot [smooth cycle] coordinates {(0:\rb) (60:\rs) (120:\rb) (180:\rs) (240:\rb) (300:\rs) };
	\end{scope}
	\begin{scope}[xshift = 0.666667cm, rotate=60]
	\draw plot [smooth cycle] coordinates {(0:\rb) (60:\rs) (120:\rb) (180:\rs) (240:\rb) (300:\rs) };
	\end{scope}
	
	\draw[opacity=0.2] ({2/3},0) -- ({4/3},0)
	($0.5*(1,{1/sqrt(3)})$) -- ({2/3},0)
	($0.5*(1,{-1/sqrt(3)})$) -- ({2/3},0)
	($(1,{1/sqrt(3)})+0.5*(1,{-1/sqrt(3)})$) -- ({4/3},0)
	($0.5*(1,{1/sqrt(3)})+(1,{-1/sqrt(3)})$) -- ({4/3},0);
	
	\begin{scope}[shift = (a)]
	\begin{scope}[xshift = 1.33333cm]
	\draw plot [smooth cycle] coordinates {(0:\rb) (60:\rs) (120:\rb) (180:\rs) (240:\rb) (300:\rs) };
	\end{scope}
	\begin{scope}[xshift = 0.666667cm, rotate=60]
	\draw plot [smooth cycle] coordinates {(0:\rb) (60:\rs) (120:\rb) (180:\rs) (240:\rb) (300:\rs) };
	\end{scope}	
	\draw[opacity=0.2] ({2/3},0) -- ({4/3},0)
	($0.5*(1,{1/sqrt(3)})$) -- ({2/3},0)
	($0.5*(1,{-1/sqrt(3)})$) -- ({2/3},0)
	($(1,{1/sqrt(3)})+0.5*(1,{-1/sqrt(3)})$) -- ({4/3},0)
	($0.5*(1,{1/sqrt(3)})+(1,{-1/sqrt(3)})$) -- ({4/3},0);
	\end{scope}
	\begin{scope}[shift = (b)]
	\begin{scope}[xshift = 1.33333cm]
	\draw plot [smooth cycle] coordinates {(0:\rb) (60:\rs) (120:\rb) (180:\rs) (240:\rb) (300:\rs) };
	\end{scope}
	\begin{scope}[xshift = 0.666667cm, rotate=60]
	\draw plot [smooth cycle] coordinates {(0:\rb) (60:\rs) (120:\rb) (180:\rs) (240:\rb) (300:\rs) };
	\end{scope}
	\draw[opacity=0.2] ({2/3},0) -- ({4/3},0)
	($0.5*(1,{1/sqrt(3)})$) -- ({2/3},0)
	($0.5*(1,{-1/sqrt(3)})$) -- ({2/3},0)
	($(1,{1/sqrt(3)})+0.5*(1,{-1/sqrt(3)})$) -- ({4/3},0)
	($0.5*(1,{1/sqrt(3)})+(1,{-1/sqrt(3)})$) -- ({4/3},0);
	\end{scope}
	\begin{scope}[shift = ($-1*(a)$)]
	\begin{scope}[xshift = 1.33333cm]
	\draw plot [smooth cycle] coordinates {(0:\rb) (60:\rs) (120:\rb) (180:\rs) (240:\rb) (300:\rs) };
	\end{scope}
	\begin{scope}[xshift = 0.666667cm, rotate=60]
	\draw plot [smooth cycle] coordinates {(0:\rb) (60:\rs) (120:\rb) (180:\rs) (240:\rb) (300:\rs) };
	\end{scope}
	\draw[opacity=0.2] ({2/3},0) -- ({4/3},0)
	($0.5*(1,{1/sqrt(3)})$) -- ({2/3},0)
	($0.5*(1,{-1/sqrt(3)})$) -- ({2/3},0)
	($(1,{1/sqrt(3)})+0.5*(1,{-1/sqrt(3)})$) -- ({4/3},0)
	($0.5*(1,{1/sqrt(3)})+(1,{-1/sqrt(3)})$) -- ({4/3},0);
	\end{scope}
	\begin{scope}[shift = ($-1*(b)$)]
	\begin{scope}[xshift = 1.33333cm]
	\draw plot [smooth cycle] coordinates {(0:\rb) (60:\rs) (120:\rb) (180:\rs) (240:\rb) (300:\rs) };
	\end{scope}
	\begin{scope}[xshift = 0.666667cm, rotate=60]
	\draw plot [smooth cycle] coordinates {(0:\rb) (60:\rs) (120:\rb) (180:\rs) (240:\rb) (300:\rs) };
	\end{scope}
	\draw[opacity=0.2] ({2/3},0) -- ({4/3},0)
	($0.5*(1,{1/sqrt(3)})$) -- ({2/3},0)
	($0.5*(1,{-1/sqrt(3)})$) -- ({2/3},0)
	($(1,{1/sqrt(3)})+0.5*(1,{-1/sqrt(3)})$) -- ({4/3},0)
	($0.5*(1,{1/sqrt(3)})+(1,{-1/sqrt(3)})$) -- ({4/3},0);
	\end{scope}
	\begin{scope}[shift = ($(a)+(b)$)]
	\begin{scope}[xshift = 1.33333cm]
	\draw plot [smooth cycle] coordinates {(0:\rb) (60:\rs) (120:\rb) (180:\rs) (240:\rb) (300:\rs) };
	\end{scope}
	\begin{scope}[xshift = 0.666667cm, rotate=60]
	\draw plot [smooth cycle] coordinates {(0:\rb) (60:\rs) (120:\rb) (180:\rs) (240:\rb) (300:\rs) };
	\end{scope}
	\draw[opacity=0.2] ({2/3},0) -- ({4/3},0)
	($0.5*(1,{1/sqrt(3)})$) -- ({2/3},0)
	($0.5*(1,{-1/sqrt(3)})$) -- ({2/3},0)
	($(1,{1/sqrt(3)})+0.5*(1,{-1/sqrt(3)})$) -- ({4/3},0)
	($0.5*(1,{1/sqrt(3)})+(1,{-1/sqrt(3)})$) -- ({4/3},0);
	\end{scope}
	\begin{scope}[shift = ($-1*(a)-(b)$)]
	\begin{scope}[xshift = 1.33333cm]
	\draw plot [smooth cycle] coordinates {(0:\rb) (60:\rs) (120:\rb) (180:\rs) (240:\rb) (300:\rs) };
	\end{scope}
	\begin{scope}[xshift = 0.666667cm, rotate=60]
	\draw plot [smooth cycle] coordinates {(0:\rb) (60:\rs) (120:\rb) (180:\rs) (240:\rb) (300:\rs) };
	\end{scope}
	\draw[opacity=0.2] ({2/3},0) -- ({4/3},0)
	($0.5*(1,{1/sqrt(3)})$) -- ({2/3},0)
	($0.5*(1,{-1/sqrt(3)})$) -- ({2/3},0)
	($(1,{1/sqrt(3)})+0.5*(1,{-1/sqrt(3)})$) -- ({4/3},0)
	($0.5*(1,{1/sqrt(3)})+(1,{-1/sqrt(3)})$) -- ({4/3},0);
	\end{scope}
	\begin{scope}[shift = ($(a)-(b)$)]
	\begin{scope}[xshift = 1.33333cm]
	\draw plot [smooth cycle] coordinates {(0:\rb) (60:\rs) (120:\rb) (180:\rs) (240:\rb) (300:\rs) };
	\end{scope}
	\begin{scope}[xshift = 0.666667cm, rotate=60]
	\draw plot [smooth cycle] coordinates {(0:\rb) (60:\rs) (120:\rb) (180:\rs) (240:\rb) (300:\rs) };
	\end{scope}
	\draw[opacity=0.2] ({2/3},0) -- ({4/3},0)
	($0.5*(1,{1/sqrt(3)})$) -- ({2/3},0)
	($0.5*(1,{-1/sqrt(3)})$) -- ({2/3},0)
	($(1,{1/sqrt(3)})+0.5*(1,{-1/sqrt(3)})$) -- ({4/3},0)
	($0.5*(1,{1/sqrt(3)})+(1,{-1/sqrt(3)})$) -- ({4/3},0);
	\end{scope}
	\begin{scope}[shift = ($-1*(a)+(b)$)]
	\begin{scope}[xshift = 1.33333cm]
	\draw plot [smooth cycle] coordinates {(0:\rb) (60:\rs) (120:\rb) (180:\rs) (240:\rb) (300:\rs) };
	\end{scope}
	\begin{scope}[xshift = 0.666667cm, rotate=60]
	\draw plot [smooth cycle] coordinates {(0:\rb) (60:\rs) (120:\rb) (180:\rs) (240:\rb) (300:\rs) };
	\end{scope}
	\draw[opacity=0.2] ({2/3},0) -- ({4/3},0)
	($0.5*(1,{1/sqrt(3)})$) -- ({2/3},0)
	($0.5*(1,{-1/sqrt(3)})$) -- ({2/3},0)
	($(1,{1/sqrt(3)})+0.5*(1,{-1/sqrt(3)})$) -- ({4/3},0)
	($0.5*(1,{1/sqrt(3)})+(1,{-1/sqrt(3)})$) -- ({4/3},0);
	\end{scope}
	\end{scope}

	\draw[dashed,opacity=0.5,->] (-3.9,0.65) .. controls(-1.8,1.5) .. (1,0.7);
	\begin{scope}[scale=2.8]	
	\coordinate (a) at (1,{1/sqrt(3)});		
	\coordinate (b) at (1,{-1/sqrt(3)});	
	\coordinate (Y) at (1.8,0.45);
	\coordinate (c) at (2,0);
	\coordinate (x1) at ({2/3},0);
	\coordinate (x0) at (1,0);
	\coordinate (x2) at ({4/3},0);

	\pgfmathsetmacro{\rb}{0.25pt}
	\pgfmathsetmacro{\rs}{0.2pt}
	
	\begin{scope}[xshift = 1.33333cm]
	\draw plot [smooth cycle] coordinates {(0:\rb) (60:\rs) (120:\rb) (180:\rs) (240:\rb) (300:\rs) };
	\draw (0:\rb) node[xshift=7pt] {$D_2$};
	\end{scope}
	\begin{scope}[xshift = 0.666667cm, rotate=60]
	\draw plot [smooth cycle] coordinates {(0:\rb) (60:\rs) (120:\rb) (180:\rs) (240:\rb) (300:\rs) };
	\end{scope}
	\draw ({0.6666667-\rb},0) node[xshift=-7pt] {$D_1$};
	
	\draw (Y) node{$Y$};
	\draw[->] (0,0) -- (a) node[above,pos=0.7]{$l_1$};
	\draw[->] (0,0) -- (b) node[below,pos=0.7]{$l_2$};
	\draw (a) -- (c) -- (b);
	\draw[fill] (x1) circle(0.5pt) node[xshift=6pt,yshift=-6pt]{$x_1$}; 
	\draw[fill] (x0) circle(0.5pt) node[yshift=4pt, xshift=6pt]{$x_0$}; 
	\draw[fill] (x2) circle(0.5pt) node[xshift=6pt,yshift=4pt]{$x_2$}; 
	\draw[dashed] (1,0.7) node[right]{$p$} -- (1,-0.7);
	\end{scope}
	\end{tikzpicture}
	\caption{Illustration of the bubbly honeycomb crystal and quantities in the fundamental domain $Y$.} \label{fig:honeycomb}
\end{figure}

We will consider a general shape of the bubbles, under certain symmetry assumptions. Let $R_0$ be the rotation around $x_0$ by $\pi$, and let $R_1$ and $R_2$ be the rotations by $-\frac{2\pi}{3}$ around $x_1$ and $x_2$, respectively. These rotations can be written as
$$ R_1 x = Rx+l_1, \quad R_2 x = Rx + 2l_1, \quad R_0 x = 2x_0 - x , $$
where $R$ is the rotation by  $-\frac{2\pi}{3}$ around the origin. Moreover, let $R_3$ be the reflection across the line $p = x_0 + \R e_2$, where $e_2$ is the second standard basis element. Assume that the unit cell contains two bubbles $D_j$, $j=1,2$, each centred at $x_j$ such that
$$R_0 D_1 = D_2, \quad R_1 D_1 = D_1,\quad R_2 D_2 = D_2, \quad R_3D_1 = D_2.$$
We denote the pair of bubbles by $D=D_1 \cup D_2$.

The dual lattice of $\Lambda$, denoted $\Lambda^*$, is generated by $\alpha_1$ and $\alpha_2$ satisfying $ \alpha_i\cdot l_j = 2\pi \delta_{ij}$, for $i,j = 1,2.$ Then 
$$ \alpha_1 = \frac{2\pi}{L}\left( \frac{1}{\sqrt{3}}, 1\right),~~\alpha_2 = \frac{2\pi}{L}\left(\frac{1}{\sqrt{3}}, -1 \right).$$

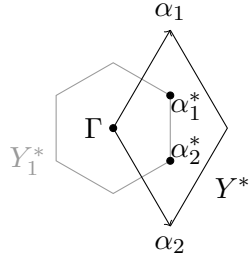
\begin{figure}
	\centering
	\begin{tikzpicture}[scale=1.3]	
	\coordinate (a) at ({1/sqrt(3)},1);		
	\coordinate (b) at ({1/sqrt(3)},-1);
	\coordinate (c) at ({2/sqrt(3)},0);
	\coordinate (K1) at ({1/sqrt(3)},{1/3});
	\coordinate (K2) at ({1/sqrt(3)},{-1/3});
	\coordinate (K3) at (0,{-2/3});
	\coordinate (K4) at ({-1/sqrt(3)},{-1/3});
	\coordinate (K5) at ({-1/sqrt(3)},{1/3});
	\coordinate (K6) at (0,{2/3});
	
	\draw[->] (0,0) -- (a) node[above]{$\alpha_1$};
	\draw[->] (0,0) -- (b) node[below]{$\alpha_2$};
	\draw (a) -- (c) -- (b) node[pos=0.4,below right]{$Y^*$};
	\draw[fill] (K1) circle(1pt) node[xshift=6pt,yshift=-4pt]{$\alpha_1^*$}; 
	\draw[fill] (K2) circle(1pt) node[xshift=6pt,yshift=4pt]{$\alpha_2^*$}; 
	\draw[fill] (0,0) circle(1pt) node[left]{$\Gamma$}; 
	
	\draw[opacity=0.4] (K1) -- (K2) -- (K3) -- (K4) node[left]{$Y_1^*$} -- (K5) -- (K6) -- cycle; 
	\end{tikzpicture}
%\end{wrapfigure}
\caption{Illustration of the dual lattice and the Brillouin zone $Y^*$.}
\label{fig:bz}
\end{figure}

The Brillouin zone $Y^*$ is defined as the torus $Y^*:= \bigslant{\R^2}{\Lambda^*}$ and can be represented either as  
$$Y^* \simeq \left\{ s \alpha_1+ t \alpha_2 ~|~ 0 \le s,t \le 1 \right\}, $$
or as the first Brillouin zone $Y_1^*$ illustrated in Figure \ref{fig:bz}.
The points $$\alpha_1^*= \frac{2\alpha_1+\alpha_2}{3}, \quad \alpha^*_2 = \frac{\alpha_1+2\alpha_2}{3} ,$$ in the Brillouin zone are called \emph{Dirac points}. For simplicity, in this work we will only consider the analysis around the Dirac point $\alpha^* := \alpha^*_1$, the main difference around $\alpha^*_2$ is summarised in Remark \ref{rmk:alpha2}.

We will denote the density and bulk modulus of the bubble by $\rho_b$ and $\kappa_b$, respectively. The corresponding parameters of the surrounding medium is denoted by $\rho$ and $\kappa$. Wave propagation in the bubbly honeycomb crystal is described by the following $\alpha$-quasi-periodic Helmholtz problem in $Y$:
\begin{equation}  \label{HP1}
\left\{
\begin{array} {lll}
&\ds \nabla \cdot \frac{1}{\rho} \nabla  u+ \frac{\omega^2}{\kappa} u  = 0 \quad &\text{in} \ Y \backslash D, \\
\nm
&\ds \nabla \cdot \frac{1}{\rho_b} \nabla  u+ \frac{\omega^2}{\kappa_b} u  = 0 \quad &\text{in} \ D, \\
\nm
&\ds  u_{+} -u_{-}  =0   \quad &\text{on} \ \partial D, \\
\nm
& \ds  \frac{1}{\rho} \frac{\partial u}{\partial \nu} \bigg|_{+} - \frac{1}{\rho_b} \frac{\partial u}{\partial \nu} \bigg|_{-} =0 \quad &\text{on} \ \partial D, \\ 
\nm
& u(x+l)= e^{\iu\alpha\cdot l} u(x) \quad & \text{for all} \ l\in \Lambda.
\end{array}
\right.
\end{equation}
Here, $\partial/\partial \nu$ denotes the normal derivative on $\partial D$, and the subscripts $+$ and $-$ indicate the limits from outside and inside $D$, respectively. A non-trivial solution to this problem and its corresponding frequency is called a Bloch eigenfunction and a Bloch eigenfrequency. The Bloch eigenfrequencies $\omega_i^\alpha, \ i=1,2,...$ with positive real part, seen as functions of $\alpha$, are known as \emph{band functions}.

Let 
\begin{equation*} \label{v}  v:=\sqrt{\frac{\kappa}{\rho}}, \ ~  v_b:=\sqrt{\frac{\kappa_b}{\rho_b}}, \ ~k=\frac{\omega}{v}, \ ~  k_b=\frac{\omega}{ v_b}.\end{equation*}
Introduce the density contrast parameter $\delta$ as 
\begin{equation*} \label{delta} 
\delta:=\frac{\rho_b}{\rho} . \end{equation*}
We assume that there is a high contrast in the density while the wave speeds are comparable, \ie{}, 
\begin{equation*} 
\delta  \ll 1 \quad \text{and} \quad v,  v_b=\O(1).
\end{equation*}

%%%%%%%%%%%%%%%%%%%%%%%%%%%%%%%%%%%%%%%%%%%%%%%%%%%%%%%%%%%%%%%%%%%%%%%%%%%%%%%%%%%%%
\subsection{Quasi-periodic Green's function for the honeycomb lattice} \label{sec:setup_G}
Define the $\alpha$-quasi-periodic Green's function $G^{\alpha,k}$ to satisfy
\begin{equation*}
\Delta G^{\alpha,k} + k^2G^{\alpha,k} = \sum_{n \in \Lambda} \delta(x-n)e^{\iu\alpha\cdot n}.
\end{equation*}
Then it can be shown that $G^{\alpha,k}$ is given by  \cite{surv235,akl} 
\begin{equation} \label{eq:green}
G^{\alpha,k}(x)= \frac{1}{|Y|}\sum_{q\in \Lambda^*} \frac{e^{\iu(\alpha+q)\cdot x}}{ k^2-|\alpha+q|^2}.
\end{equation}

For a given bounded domain $D$ in  $Y$, with Lipschitz boundary $\partial D$, the single layer potential $\Scal_D^{\alpha,k}: L^2(\partial D) \rightarrow H_{\textrm{loc}}^1(\R^2)$  is defined by
\begin{equation*}
\mathcal{S}_{D}^{\alpha,k}[\varphi](\Bx):=\int_{\partial D} G^{\alpha,k}( \Bx-\By ) \varphi(\By) \dx \sigma(\By),~~~\Bx\in \mathbb{R}^{2}.
\end{equation*}
Here, we denote by $H_{\textrm{loc}}^1(\R^2)$ the space of functions that are square integrable on every compact subset of $\R^2$ and have a weak first derivative that is also square integrable. The following jump relations are well-known \cite{surv235,akl}:
\begin{equation*} \label{jump1}
\left. \frac{\partial}{\partial\nu} \mathcal{S}_{D}^{\alpha,k}[\varphi]\right|_{\pm}(\Bx)=\left (\pm\frac{1}{2}I+(\mathcal{K}_{D}^{-\alpha,k})^*\right)[\varphi](\Bx),~~~\Bx\in \partial D,
\end{equation*}
where the Neumann-Poincar\'e operator $(\mathcal{K}_{D}^{-\alpha,k})^*: L^2(\p D) \rightarrow L^2(\p D)$ is defined as
\begin{equation*}
(\mathcal{K}_{D}^{-\alpha,k})^*[\varphi](x)=\text{p.v.}\int_{\partial D}\frac {\p }{\p \nu_x}G^{\alpha,k}(\Bx -\By) \varphi(\By) \dx \sigma(\By),~~~\Bx\in \partial D.
\end{equation*}
The Green's function can be asymptotically expanded for small $k$ as follows \cite{surv235,honeycomb}:
\begin{equation*}
G^{\alpha,k}= G^{\alpha,0}+ k^{2} G_1^{\alpha} + \O(k^4), \qquad G_1^{\alpha}(x):=-\sum_{q\in\Lambda^*}\frac{e^{\iu(\alpha+q)\cdot x}}{ |\alpha+q|^{2}},
\end{equation*}
where the error term is uniform in $\alpha$ in a neighbourhood of $\alpha^*$ and $x\in Y$. This leads to the following expansion of the single layer potential $\Scal_D^{\alpha,k}: L^2(\partial D) \rightarrow H^1(\p D)$ 
\begin{equation} \label{eq:expS}
\Scal_D^{\alpha,k} = \Scal_D^{\alpha,k} + k^{2} \Scal^{\alpha}_{D,1} + \O(k^4), \quad \Scal^{\alpha}_{D,1}[\phi](x) := \int_{\p D} G_1^{\alpha}(x-y) \phi(y) \dx \sigma(y),
\end{equation}
where $\O(k^4)$ denotes an operator $ L^2(\partial D) \rightarrow H^1(\p D)$ with operator norm of order $k^4$, uniformly for $\alpha$ in a neighbourhood of $\alpha^*$. Similarly, for the Neumann-Poincar\'e operator, we have
\begin{equation} \label{eq:expK}
(\Kcal_D^{-\alpha,k})^* = (\Kcal_D^{-\alpha,0})^* + k^{2} \mathcal{K}^{\alpha}_{D,1} + \O(k^4), \quad \mathcal{K}^{\alpha}_{D,1} [\phi](x) := \int_{\p D} \frac{\p }{\p \nu_x}G_1^{\alpha}(x-y) \phi(y) \dx \sigma(y),
\end{equation} 
where $\O(k^4)$ denotes an operator $ L^2(\partial D) \rightarrow L^2(\partial D)$ with operator norm of order $k^4$, uniformly for $\alpha$ in a neighbourhood of $\alpha^*$.

It is known that $\Scal_D^{\alpha,0}: L^2(\partial D) \rightarrow H^1(\p D)$ is invertible when $\alpha \neq 0$ \cite{surv235,akl}. Let $\psi_j^{\alpha}\in L^2(\p D)$  be given by
\begin{align} \label{psi_def}
\mathcal{S}_D^{\alpha,0}[\psi_j^{\alpha}] = \chi_{\partial D_j}\quad \mbox{on}~\partial D,\quad j=1, 2,
\end{align}
where $\chi$ denotes the indicator function. Define the capacitance coefficient matrix $C^\alpha=(C_{ij}^\alpha)$ by
$$ C_{ij}^\alpha := - \int_{\partial D_i} \psi_j^\alpha \dx \sigma,\quad i,j=1, 2.$$
Using the symmetry of the honeycomb structure, it was shown in \cite{honeycomb} that the capacitance coefficients satisfy 
$$c_1^\alpha := C_{11}^\alpha = C_{22}^{\alpha}, \quad c_2^\alpha := C_{12}^{\alpha} = \overline{C_{21}^{\alpha}},$$ 
and
\begin{equation}\label{c1c2_deri}
\nabla_\alpha c_1^\alpha \Big|_{\alpha=\alpha^*} = 0,
\quad
\nabla_\alpha c_2^\alpha \Big|_{\alpha=\alpha^*} = c\begin{pmatrix} 1\\-\iu\end{pmatrix},
\end{equation}
where we denote $$c:=\pd{c_2^{\alpha}}{\alpha_1}\Big|_{\alpha=\alpha^*}.$$
In \cite[Lemma 3.4]{honeycomb}, it was shown that $c\neq0$.
Note that the capacitance matrix $C^\alpha$ is written as
\begin{equation*}
C^\alpha = \begin{pmatrix}
c_1^\alpha & c_2^\alpha \\
\overline{c_2^\alpha} & c_1^\alpha
\end{pmatrix}.
\end{equation*}

\subsection{ Dirac cone dispersion in the band structure}

The solution to \eqnref{HP1} can be represented using
the single layer potentials $\Scal_D^{\alpha, k_b}$ and $\Scal_D^{\alpha,k}$ as 
follows (see, for instance, \cite{akl}):
\begin{equation}\label{eq:u}
u(\Bx) = \begin{cases}
\Scal_D^{\alpha, k_b} [\phi](\Bx),\quad \Bx\in D, \\
\Scal_D^{ \alpha,k}[\psi](\Bx),\quad \Bx\in Y\setminus\overline{D},
\end{cases}
\end{equation}
where the pair $(\phi,\psi)\in L^2(\p D)\times L^2(\p D)$ satisfies
\begin{equation} \label{int_eq12}
\mathcal{A}_\delta^{\alpha,\omega} \begin{pmatrix} \phi \\ \psi \end{pmatrix}=0.
\end{equation}
Here, the operator $\Acal_\delta^{\alpha,\omega}$ is defined by
\begin{equation}\label{eq:A}
\Acal_\delta^{\alpha,\omega} := 
\begin{bmatrix}
	\Scal^{\alpha, k_b}_D & -\Scal^{\alpha, k}_D\\
	-\frac{1}{2}I+(\mathcal{K}_{D}^{-\alpha, k_b})^*  & -\delta \left (\frac{1}{2}I+(\mathcal{K}_{D}^{-\alpha,k})^*\right) 
\end{bmatrix}.
\end{equation}
It is well-known that the integral equation \eqref{int_eq12} has a non-trivial solution 
for some discrete frequencies $\omega$, and the frequencies with positive real part are the band functions $\omega_i^\alpha, \ i=1,2,...$

It was shown in \cite{honeycomb} that, for the honeycomb structure, the first two band functions $\omega_1^\alpha$ and $\omega_2^\alpha$ form a conical dispersion relation near the Dirac point $\alpha^*$. Such a conical dispersion is referred to as a {\it Dirac cone}. More specifically, the following theorem was proved in \cite{honeycomb} (Theorems 3.2 and 4.1 from \cite{honeycomb}). It is worth emphasizing that the following results hold in the deep sub-wavelength regime. 
\begin{theorem}\label{thm:honeycomb}
	For small $\delta$, the first two band functions $\omega_j^\alpha,j=1,2$, satisfy 
	\begin{equation} \label{eq:wasymp}
	\omega_j^\alpha = \sqrt{\frac{\delta\lambda_j^\alpha }{|D_1|}}v_b + \O(\delta),
	\end{equation}
	uniformly for $\alpha$ in a neighbourhood of $\alpha^*$, where $\lambda_j^\alpha, j=1,2,$ are the two eigenvalues of $C^\alpha$ and $|D_1|$ denotes the area of one of the bubbles. Moreover, for $\alpha$ close to $\alpha^*$ and $\delta$ small enough, the first two band functions form a Dirac cone, \ie{},
	\begin{equation} \label{eq:dirac}
	\begin{matrix}
	\ds \omega_1^\alpha = \omega^*- \lambda|\alpha - \alpha^*| \big[ 1+ \O(|\alpha-\alpha^*|) \big], \\[0.5em]
	\ds \omega_2^\alpha = \omega^*+ \lambda|\alpha - \alpha^*| \big[ 1+ \O(|\alpha-\alpha^*|) \big],
	\end{matrix}
	\end{equation}
	where $\omega^*$ and $\lambda$ are independent of $\alpha$ and satisfy
	$$\omega^*= \sqrt{\frac{\delta c_1^{\alpha^*}}{|D_1|}}v_b + \O(\delta) \quad \text{and} \quad \lambda =  |c|\sqrt\delta\lambda_0 + \O(\delta),  \quad \lambda_0=\frac{1}{2}\sqrt{\frac{v_b^2 }{|D_1|c_1^{\alpha^*}}}$$
	as $\delta \rightarrow 0$. Moreover, the error term $\O(|\alpha-\alpha^*|)$ in (\ref{eq:dirac}) is uniform in $\delta$.
\end{theorem}
In the next sections, we will investigate the asymptotic behaviour of the  Bloch eigenfunctions  near the Dirac points. Then we shall prove that the envelopes of the Bloch eigenfunctions satisfy a Helmholtz equation with near-zero effective refractive index and  derive a two-dimensional homogenized equation of Dirac-type for the honeycomb bubbly crystal.

\section{Bloch eigenfunctions near Dirac points} \label{sec:efunc}
In this section, we study the Bloch eigenfunctions in the regime close to $\alpha^*$, following the approach of \cite{homogenization}. We will perform an asymptotic analysis with two small parameters. For $f,g\in C(\R^2,\R)$, depending on the small parameters $\epsilon_1,\epsilon_2 \in \R$, we will use the notation
$$f(\epsilon_1,\epsilon_2) = \O(g(\epsilon_1, \epsilon_2))$$
to denote that there is some $K > 0$, constant in $\epsilon_1, \epsilon_2$, such that $|f(\epsilon_1,\epsilon_2)| < K|g(\epsilon_1, \epsilon_2)|$ for all $(\epsilon_1, \epsilon_2)$ in some neighbourhood of $(0,0)$.

\subsection{Asymptotic behaviour of the single layer potential}
Here we consider the single layer potential $\Scal_D^{\alpha,k}$ near a Dirac point $\alpha = \alpha^*$.
\begin{lemma}\label{lem:S_asymp}
	Let $\psi_j^\alpha$, for $j=1,2,$ be defined by (\ref{psi_def}). For $j=1,2$ and for $\alpha$ near a Dirac point $\alpha^*$, \ie{} $\alpha = \alpha^* + \epsilon \tilde\alpha$ with small $\epsilon >0$ and fixed $\tilde\alpha$, we have
	\begin{equation} \label{eq:S_asymp}
	\Scal_D^{\alpha^* + \epsilon \tilde{\alpha},k}\left[\psi_j^{\alpha^* + \epsilon \tilde{\alpha}}\right](x) = e^{\iu\epsilon \tilde{\alpha}\cdot x}\Scal_D^{\alpha^*,0}\big[\xi^{\epsilon\tilde{\alpha}}_j\big](x) + \O(\epsilon^2 + \epsilon k^2), \qquad x\in \mathbb{R},
	\end{equation}
	where $$\xi^{\epsilon\tilde{\alpha}}_j := \left(\Scal_D^{\alpha^*,0}\right)^{-1}\left[e^{-\iu\epsilon\tilde{\alpha}\cdot y}\chi_{\partial D_j}(y)\right].$$
\end{lemma}
\begin{proof}
As in \cite{honeycomb}, we have the Taylor expansion 
\begin{align*}
G^{\alpha^*+\epsilon\tilde{\alpha},k}(x) &= G^{\alpha^*,k}(x)+ \frac{1}{|Y|}\sum_{q\in\Lambda^*} \frac{e^{\iu(\alpha^*+q)\cdot x}}{ k^2-|\alpha^*+q|^2}  \left( \iu\epsilon\tilde\alpha\cdot x + 2\frac{(\alpha^*+q)\cdot\epsilon\tilde\alpha}{k^2-|\alpha^*+q|^2}\right)+\O(\epsilon^2) \\
&= e^{\iu\epsilon\tilde{\alpha}\cdot x}\left(G^{\alpha^*,k}(x) + G_1^{\epsilon\tilde{\alpha},k}(x)\right) + \O(\epsilon^2),
\end{align*}
uniformly for $k$ in a neighbourhood of $0$, where
$$ G_1^{\epsilon\tilde\alpha, k}(x) =  \frac{2}{|Y|}\sum_{q\in\Lambda^*} \frac{e^{\iu(\alpha^*+q)\cdot x}(\alpha^*+q)\cdot\epsilon\tilde\alpha}{ \left(k^2-|\alpha^*+q|^2\right)^2}.$$
We define the corresponding operator $\Scal_1^{\epsilon\tilde\alpha, k}: L^2(\p D) \rightarrow H^1(\p D)$ as
$$\Scal_1^{\epsilon\tilde\alpha, k}[\varphi](x) = \int_{\p D} G_1^{\epsilon\tilde\alpha,k}(x-y) \varphi(y) \dx \sigma(y).$$
Observe that, in the operator norm, $\Scal_1^{\epsilon\tilde\alpha, k} = \O(\epsilon)$, uniformly in $k$. With these definitions, we have the following expansion of the single layer potential for $\alpha$ close to $\alpha^*$:
\begin{equation}
\Scal_D^{\alpha^* + \epsilon\tilde{\alpha},k}[\varphi](x) = e^{\iu\epsilon \tilde\alpha\cdot x}\left(\Scal_D^{\alpha^*,k} + \Scal_1^{\epsilon\tilde\alpha, k}\right)\left[e^{-\iu\epsilon\tilde{\alpha}\cdot y}\varphi(y)\right](x) + \O(\epsilon^2), \label{eq:S_exp}
\end{equation}
uniformly in $k$. Using the Neumann series, we can then compute
\begin{align}
\psi^{\alpha^* + \epsilon\tilde{\alpha}}_j(y) &= \left( \Scal_D^{\alpha^*+\epsilon\tilde{\alpha},0}\right)^{-1}\left[\chi_{\partial D_j}\right](y) \nonumber \\ 
&= e^{\iu\epsilon\tilde{\alpha}\cdot y}\left(I - \left(\Scal_D^{\alpha^*,0}\right)^{-1} \Scal_1^{\epsilon\tilde{\alpha},0}\right)\left(\Scal_D^{\alpha^*,0}\right)^{-1}\left[e^{-\iu\epsilon\tilde{\alpha}\cdot x}\chi_{\partial D_j}(x)\right](y) + \O(\epsilon^2) \nonumber \\
&= e^{\iu\epsilon\tilde{\alpha}\cdot y}\left(I - \left(\Scal_D^{\alpha^*,0}\right)^{-1} \Scal_1^{\epsilon\tilde{\alpha},0}\right)\big[\xi^{\epsilon\tilde{\alpha}}_j\big](y) + \O(\epsilon^2),\qquad y \in \p D. \label{eq:psi_exp}
\end{align}
Then, combining \eqnref{eq:S_exp} and \eqnref{eq:psi_exp} and using the fact that $\Scal_D^{\alpha,k} = \Scal_D^{\alpha,0} +\O(k^2)$, uniformly for $\alpha$ in a neighbourhood of $\alpha^*$, we obtain that
\begin{align*}
\Scal_D^{\alpha^* + \epsilon \tilde{\alpha},k}\left[\psi_j^{\alpha^* + \epsilon \tilde{\alpha}}\right](x) &= e^{\iu\epsilon \tilde\alpha\cdot x}\left(\Scal_D^{\alpha^*,k} + \Scal_1^{\epsilon\tilde\alpha, k} - \Scal_D^{\alpha^*,k}\left(\Scal_D^{\alpha^*,0}\right)^{-1} \Scal_1^{\epsilon\tilde{\alpha},0} \right)\left[\xi^{\epsilon\tilde{\alpha}}_j\right](x) + \O(\epsilon^2) \\
&= e^{\iu\epsilon \tilde\alpha\cdot x}\Scal_D^{\alpha^*,0}\big[\xi^{\epsilon\tilde{\alpha}}_j\big](x) + \O(\epsilon^2 + \epsilon k^2),
\end{align*}
which concludes the proof.
\end{proof}
\begin{remark} \label{rmk:xi}
	From the definition of $\xi^{\epsilon\tilde{\alpha}}_j$, we see that 
	\begin{equation} \label{eq:xi}
	\xi^{\epsilon\tilde{\alpha}}_j = \psi^{\alpha^*}_j + \O(\epsilon).
	\end{equation}
	Hence, equation \eqnref{eq:S_asymp} contains two terms of order $\epsilon$: one term from the Taylor expansion of $e^{i\epsilon\tilde{\alpha}\cdot x}$ and one term from the Taylor expansion of $e^{i\epsilon\tilde{\alpha}\cdot y}$ inside $\xi^{\epsilon\tilde{\alpha}}_j$. Here, $x$ and $y$ varies on different scales: $x\in \R^2$ while $y\in \p D$. Since $y$ varies on a much smaller scale, the latter term will be negligible. This will be made precise in Section \ref{sec:hom}.
\end{remark}

\begin{remark}
	Equation \eqnref{eq:S_exp} differs from \cite[Lemma 4.4]{homogenization} by the $\epsilon$-order term $\Scal_1$ and the factor $e^{-\iu\epsilon \tilde \alpha\cdot y}$, which, in general, do not cancel. This is due to small miscalculations in \cite{homogenization}. Nevertheless, under the homogenization setting defined in Section \ref{sec:hom} and \cite[Section 5]{homogenization}, these factors only contributes to the higher orders in the expansions, and does not affect the final theorems. Therefore, despite the miscalculations in \cite{homogenization}, the main result, \cite[Theorem 5.4]{homogenization} is still correct.
	% However, the factor $e^{-i\epsilon \tilde \alpha\cdot y}$ means that we have to introduce $\xi^{\epsilon\tilde{\alpha}}_j$ as a small correction to $\psi^{\alpha^*}_j$, as discussed in Remark \ref{rmk:xi}.
\end{remark}

\subsection{Bloch eigenfunctions near the Dirac points}

Here, we consider the asymptotic behaviour of the Bloch eigenfunctions near the Dirac points.

Let us assume that $\alpha$ is close to the Dirac point $\alpha^*$, \ie{}, $\alpha = \alpha^* + \epsilon \tilde{\alpha}$ for small $\epsilon>0$. Let $u^\alpha$ be the Bloch eigenfunction with the Bloch eigenfrequency $\omega^\alpha$. In other terms, $u^\alpha$ is given by
$$
u^\alpha(\Bx) = \begin{cases}
	\Scal_D^{\alpha, k_b} [\phi^\alpha](\Bx),\quad \Bx\in D, \\
	\Scal_D^{ \alpha,k}[\psi^\alpha](\Bx),\quad \Bx\in Y\setminus\overline{D},
\end{cases}
$$
where the pair $(\phi^\alpha,\psi^\alpha)\in L^2(\p D)\times L^2(\p D)$ satisfies
\begin{equation} \label{int_eq123}
\mathcal{A}_\delta^{\alpha,\omega^\alpha} \begin{pmatrix} \phi^\alpha \\ \psi^\alpha \end{pmatrix}=0.
\end{equation}

We know from Theorem \ref{thm:honeycomb}  that $\omega^\alpha = \O(\sqrt{\delta})$, uniformly in $\epsilon$. So, for $\delta$ small enough and using \eqnref{eq:expS} and \eqnref{eq:expK}, the integral equation \eqnref{int_eq123} can be approximated by
\begin{align}
\begin{cases}
\Scal_D^{\alpha,0}[\phi^\alpha]- \Scal_D^{\alpha, 0}[\psi^\alpha] = \O(\delta),
\\[0.5em]
\left(-\frac{1}{2}I + (\Kcal_D^{-\alpha,0})^* + k_b^2 \Kcal_{D,1}^{\alpha}\right)[\phi^\alpha]-\delta \left(\frac{1}{2}I+(\mathcal{K}_{D}^{-\alpha,0})^*\right)[\psi^\alpha]= \O(\delta^2),
\end{cases}
\label{eq_line2}
\end{align}
uniformly for $\alpha$ in a neighbourhood of $\alpha^*$. Since $\mathcal{S}_D^{\alpha,0}$ is invertible, and the inverse is bounded for $\alpha$ in a neighbourhood of $\alpha^*$, the first equation in \eqnref{eq_line2} implies
\begin{equation}  \label{eq_line2-a}
\psi^\alpha= \phi^\alpha+ \O(\delta),
\end{equation}
uniformly in $\alpha$. Substituting \eqnref{eq_line2-a} into the second equation in \eqnref{eq_line2}, we have
\begin{equation}
\left(-\frac{1}{2}I + (\Kcal_D^{-\alpha,0})^* + k_b^2 \Kcal_{D,1}^{\alpha}\right)[\phi^\alpha]-\delta \left(\frac{1}{2}I+(\mathcal{K}_{D}^{-\alpha,0})^*\right)[\phi^\alpha]= \O(\delta^2), \label{int_eq_redc}
\end{equation}
uniformly in $\alpha$. Since
$ \ker \left(-\frac{1}{2}I + (\Kcal_D^{-\alpha,0})^*\right)$ is generated by $\psi_1^\alpha$ and $\psi_2^\alpha$, which are defined by (\ref{psi_def}), 
 we may write $\phi^\alpha$ as
\begin{equation*}
\phi^\alpha= A \psi_1^\alpha + B\psi_2^\alpha + \O(\delta),\label{phi_approx}
\end{equation*}
uniformly in $\alpha$, where $|A|+|B|=1$.

By integrating \eqnref{int_eq_redc} on $\partial D_1$ and $\partial D_2$, and using the following identity \cite{surv235}
\begin{equation*} \label{eq:K1int}
\int_{\p D_j} \Kcal^{\alpha}_{D,1} [\phi] \dx \sigma  = -\int_{D_j}\Scal_D^{\alpha,0} [\phi] \dx x,\qquad j = 1,2,
\end{equation*}
it follows that
\begin{align*}
-\frac{(\omega^\alpha)^2 |D_1|}{v_b^2} A +\delta (A c_{1}^\alpha+B c_{2}^\alpha)= \O(\delta^2),\\
-\frac{(\omega^\alpha)^2 |D_2|}{v_b^2} B +\delta (A \overline{c_{2}^\alpha}+B c_{1}^\alpha)= \O(\delta^2),
\end{align*}
uniformly for $\alpha$ in a neighbourhood of $\alpha^*$. Observe that $|D_1| = |D_2|$. Thus, since we have from \eqref{c1c2_deri} that
$$
c_1^{\alpha^*+\epsilon\tilde\alpha} = c_1^{\alpha^*} + \O(\epsilon^2)
,\quad
c_2^{\alpha^*+\epsilon\tilde\alpha} = \epsilon c(\tilde\alpha_1 - \iu\tilde\alpha_2) + \O(\epsilon^2),
$$
we obtain
\begin{align*}
-\frac{|D_1|}{v_b^2} \left((\omega^\alpha)^2 - (\omega^*)^2 \right)A + {c}\delta\epsilon \Big( \tilde\alpha_1 - \iu \tilde\alpha_2\Big) B= \O(\delta\epsilon^2 + \delta^2),\\
-\frac{|D_1|}{v_b^2} \left((\omega^\alpha)^2 - (\omega^*)^2 \right)B + \overline{c}\delta\epsilon \Big( \tilde\alpha_1 + \iu \tilde\alpha_2\Big) A= \O(\delta\epsilon^2 + \delta^2).
\end{align*}
From \eqnref{eq:dirac} we know that  
$$\omega^\alpha + \omega^* = 2\omega^* + \O(\epsilon\sqrt{\delta}), \quad \omega^\alpha - \omega^* = \O(\epsilon\sqrt{\delta}),$$
and from \eqnref{eq:wasymp} that 
$$\omega^*= \sqrt{\frac{\delta c_1^{\alpha^*} v_b^2}{|D_1|}} + \O(\delta),$$
so we arrive at
\begin{align*}
-2\sqrt{\frac{\delta c_1^{\alpha^*} |D_1|}{v_b^2}} (\omega^\alpha-\omega^*)A + {c}\delta\epsilon \Big( \tilde\alpha_1 - \iu \tilde\alpha_2\Big) B = \O(\delta\epsilon^2 + \delta^2),\\
-2\sqrt{\frac{\delta c_1^{\alpha^*} |D_1|}{v_b^2}} (\omega^\alpha-\omega^*)B + \overline{c}\delta\epsilon \Big( \tilde\alpha_1 + \iu \tilde\alpha_2\Big) A = \O(\delta\epsilon^2 + \delta^2).
\end{align*}
In matrix form, this reads
\begin{align}
\lambda_0 \sqrt{\delta}
\begin{bmatrix}
 0 & \epsilon c( \tilde{\alpha}_1 - \iu \tilde{\alpha}_2)
 \\
 \epsilon \overline{c}( \tilde{\alpha}_1 + \iu \tilde{\alpha}_2) & 0
\end{bmatrix}
\begin{bmatrix}
A
\\
B
\end{bmatrix}
= (\omega^\alpha-\omega^*)
\begin{bmatrix}
A
\\
B
\end{bmatrix} + \O(\delta^{1/2} \epsilon^2 + \delta^{3/2}),
\label{pre_Dirac_eq}
\end{align}
where, as in Theorem \ref{thm:honeycomb},
$$
\lambda_0 =  \frac{1}{2}\sqrt{\frac{  v_b^2}{c_1^{\alpha^*} |D_1|}}.
$$
Then, by solving the above eigenvalue problem, we obtain two (approximate) eigenpairs given by
$$
\omega^\alpha_\pm  = \omega^* \pm \sqrt{\delta} \lambda_0 \epsilon|c|\cdot |\tilde{\alpha}| + \O(\delta^{1/2} \epsilon^2 + \delta^{3/2}),
$$
and
\begin{equation} \label{abpm}
\begin{bmatrix}
A_\pm
\\
B_\pm 
\end{bmatrix}
=
\frac{1}{\sqrt{2}} \begin{bmatrix}
\ds\pm \frac{c}{|c|} \frac{\tilde{\alpha}_1 - \iu \tilde{\alpha}_2}{|\tilde{\alpha}|}
\\[1.em]
1
\end{bmatrix} + \O(\delta + \epsilon^2).
\end{equation}

This implies that the Bloch eigenfunction  $u_+^\alpha$ (respectively, $u_-^\alpha$) associated to the upper part $\omega_+^\alpha$ (respectively, the lower part $\omega_-^\alpha$) of the Dirac cone can be represented, for $x$ outside $D$, as
$$
u_\pm^{\alpha} = A_\pm \mathcal{S}_D^{\alpha,\omega_\pm^\alpha/v} [\psi_1^\alpha] + B_\pm \mathcal{S}_D^{\alpha,\omega_\pm^\alpha/v} [\psi_2^\alpha] + \O(\delta),
$$
uniformly for $\alpha$ in a neighbourhood of $\alpha^*$. Then, by Lemma \ref{lem:S_asymp}, equation \eqnref{eq:expS} and the fact that $\omega_\pm^\alpha= \O(\delta^{1/2})$ uniformly for $\alpha$ in a neighbourhood of $\alpha^*$, we have
\begin{equation}\label{eq:Bloch_eigen_asymp}
u^{\alpha^* + \epsilon \tilde\alpha}_\pm(x) =  A_\pm e^{\iu\epsilon\tilde\alpha \cdot x} \mathcal{S}_D^{\alpha^*,0} [\xi_1^{\epsilon\tilde{\alpha}}](x) + B_\pm e^{\iu\epsilon\tilde\alpha \cdot x} \mathcal{S}_D^{\alpha^*,0} [\xi_2^{\epsilon\tilde{\alpha}}](x) + \O(\delta + \epsilon^2).
\end{equation}

\section{Homogenization of the Bloch eigenfunctions near the Dirac points} \label{sec:hom}

Here, we consider the rescaled bubbly honeycomb crystal by replacing the lattice constant $L$ with $sL$ where $s>0$ is a small positive parameter. We then derive a homogenized equation.

We need the following lemma which can be proved by a scaling argument.

\begin{lemma}\label{lem:scale}
Let $\omega_j^\alpha,j=1,2$,  be the first two eigenvalues  and $u_j^\alpha$ be the associated Bloch eigenfunctions for the bubbly honeycomb crystal with lattice constant $L$. Then, the bubbly honeycomb crystal with lattice constant $sL$ has the first two Bloch eigenvalues
$$
\omega_{\pm,s}^{\alpha/s} = \frac{1}{s} \omega_\pm^\alpha,
$$
and the corresponding eigenfunctions are 
$$
u^{\alpha/s}_{\pm,s}(x) = u_\pm^\alpha\left( \frac{x}{s}\right).
$$
\end{lemma}

We see from the above lemma that the Dirac cone is located at the point $\alpha^*/s$. 
We denote the Dirac frequency by
$$
\omega_s^* = \frac{1}{s}\omega^*.
$$
In the sequel, in order to simplify the presentation, we assume the following:
\begin{asump} \label{assump:v}
	The wave speed inside the bubble is equal to the one outside, \ie,
	$$
	v = v_b=1.
	$$
\end{asump}
Then the wave numbers $k$ and $k_b$ become
$$
k=\omega^\alpha, \quad k_b=\omega^\alpha.
$$
We have the following result for the Bloch eigenfunctions $u_{j,s}^{\alpha/s},j=1,2,$ for $\alpha/s$ near the Dirac points $\alpha^*/s$.

\begin{lemma}\label{lem:Bloch_scale}
We have
$$
u_{\pm,s}^{\alpha^*/s+\tilde{\alpha}}(x) = A_\pm e^{\iu\tilde\alpha \cdot x} S_{1} \left(\frac{x}{s}\right) + B_\pm e^{\iu\tilde\alpha \cdot x} S_{2} \left(\frac{x}{s}\right) + \O(\delta + s),
$$
where
$$
S_j (x) = \mathcal{S}_D^{\alpha^*,0} [\psi_j^{\alpha^*}](x), \quad j=1,2.
$$
\end{lemma}
\begin{proof}
We know that $u_{\pm,s}^{\alpha^*/s+\tilde{\alpha}}(x) = u_\pm^{\alpha^* + s \tilde{\alpha}}(x/s)$, which together with  \eqref{eq:Bloch_eigen_asymp} gives 
$$u_{\pm,s}^{\alpha^*/s+\tilde{\alpha}}(x) =  A_\pm e^{\iu\tilde\alpha \cdot x} \mathcal{S}_D^{\alpha^*,0} \big[\xi_1^{s\tilde{\alpha}}\big]\left(\frac{x}{s}\right) + B_\pm e^{\iu\tilde\alpha \cdot x} \mathcal{S}_D^{\alpha^*,0} \big[\xi_2^{s\tilde{\alpha}}\big]\left(\frac{x}{s}\right) + \O(\delta + s^2).$$
By Assumption \ref{assump:v}, this is valid both outside and inside $D$. Finally, from \eqnref{eq:xi} we have
$$
\mathcal{S}_D^{\alpha^*,0} \big[\xi_j^{s\tilde{\alpha}}\big](x) = S_j (x) + \O(s), \quad j=1,2,
$$
from which the conclusion follows.
\end{proof}

We see that the functions $S_1$ and $S_2$ describe the microscopic behaviour of the Bloch eigenfunction $u_{\pm,s}^{\alpha^*/s+\tilde{\alpha}}$ while $A_\pm e^{\iu\tilde\alpha\cdot x}$ and  $B_\pm e^{\iu\tilde\alpha\cdot x}$ describe the macroscopic behaviour.
Now, we derive a homogenized equation near the Dirac frequency $\omega_s^*$.

Recall that the Dirac frequency of the unscaled honeycomb crystal satisfies $\omega^* = \O(\sqrt\delta) $.
As in \cite{homogenization}, in order to make the order of $\omega_s^*$ fixed when $s$ tends to zero, we assume that
\begin{asump}
$\delta = \mu s^2$
for some fixed $\mu>0$. 
\end{asump}
Then we have
$$
\omega_s^* = \frac{1}{s} \omega^*= \O(1) \quad \mbox{as } s\rightarrow 0.
$$
So, in what follows, we omit the subscript $s$ in $\omega_s^*$, namely, $\omega^*:=\omega_s^*$.
Suppose the frequency $\omega$ is close to $\omega^*$, \ie{},
$$
\omega-\omega^* = \beta \sqrt\delta \quad \mbox{for some constant } \beta.
$$
We need to find the Bloch eigenfunctions or $\tilde{\alpha}$  such that
$$
\omega = \omega_{\pm,s}^{\alpha^*/s + \tilde\alpha}.
$$ 
We have from \eqref{pre_Dirac_eq} and Lemmas \ref{lem:scale} and \ref{lem:Bloch_scale} that the corresponding $\tilde{\alpha}$ satisfies
\begin{align*}
\lambda_0
\begin{bmatrix}
 0 &  c( \tilde{\alpha}_1 - \iu \tilde{\alpha}_2)
 \\
 \overline{c}( \tilde{\alpha}_1 + \iu \tilde{\alpha}_2) & 0
\end{bmatrix}
\begin{bmatrix}
A_\pm
\\
B_\pm
\end{bmatrix}
= \beta
\begin{bmatrix}
A_\pm
\\
B_\pm 
\end{bmatrix} + \O(s).
\end{align*}
So, it is immediate to see that the macroscopic field $
 [\tilde{u}_{1}, \tilde{u}_{2}]^T:=[A_\pm e^{\iu\tilde{\alpha}\cdot x}, B_\pm e^{\iu\tilde{\alpha}\cdot x}]^T$ satisfies the system of Dirac equations as follows:
$$
\lambda_0
\begin{bmatrix}
 0 & (-c\iu)( \p_1 - \iu \p_2)
 \\
 (-\overline{c}\iu)( \p_1 + \iu \p_2) & 0
\end{bmatrix}
\begin{bmatrix}
\tilde{u}_{1}
\\
\tilde{u}_{2}
\end{bmatrix}
= \beta
\begin{bmatrix}
\tilde{u}_{1}
\\
\tilde{u}_{2}
\end{bmatrix}.
$$ 
Here, the superscript $T$ denotes the transpose and $\partial_i$ is the partial derivative with respect to the $i$th variable. 
Note that the each component $\tilde{u}_j,j=1,2$, of the macroscopic field satisfies the Helmholtz equation
\begin{equation} \label{eq:hom}
\Delta \tilde{u}_j + \frac{\beta^2}{|c|^2\lambda_0^2} \tilde{u}_j = 0.
\end{equation}

The following is the main result on the homogenization theory for the honeycomb bubbly crystals.

\begin{theorem} \label{thm:main}
For frequencies $\omega$ close to the Dirac frequency $\omega^*$, namely, $\omega-\omega^* = \beta \sqrt\delta$, the following asymptotic behaviour of the Bloch eigenfunction $u^{\alpha^*/s + \tilde{\alpha}}_s$ holds:
$$
u_{s}^{\alpha^*/s+\tilde{\alpha}}(x) = A e^{\iu\tilde\alpha \cdot x} S_{1} \left(\frac{x}{s}\right) + B e^{\iu\tilde\alpha \cdot x} S_{2} \left(\frac{x}{s}\right) + \O(s),
$$  
where the macroscopic field $[\tilde{u}_{1}, \tilde{u}_{2}]^T:=[A e^{\iu\tilde{\alpha}\cdot x}, B e^{\iu\tilde{\alpha}\cdot x}]^T$ satisfies the two-dimensional Dirac equation
$$
\lambda_0
\begin{bmatrix}
 0 & (-c\iu)( \p_1 - \iu \p_2)
 \\
 (-\overline{c}\iu)( \p_1 + \iu \p_2) & 0
\end{bmatrix}
\begin{bmatrix}
\tilde{u}_{1}
\\
\tilde{u}_{2}
\end{bmatrix}
= \frac{\omega-\omega^*}{\sqrt\delta}
\begin{bmatrix}
\tilde{u}_{1}
\\
\tilde{u}_{2}
\end{bmatrix},
$$ 
which can be considered as a homogenized equation for the honeycomb bubbly crystal while the microscopic fields $S_1$ and $S_2$ vary on the scale of $s$.
\end{theorem}
\begin{remark} \label{rmk:alpha2}
	Theorem \ref{thm:main} is valid around the Dirac point $\alpha^* = \alpha_1^*$.	Around the other Dirac point, analogous arguments show that Theorem \ref{thm:main} is valid with  $\alpha^*=\alpha_2^*$ and the macroscopic field now satisfying 
	$$
	\lambda_0
	\begin{bmatrix}
	0 & (-{c}\iu)( \p_1 + \iu \p_2)
	\\
	(-\overline{c}\iu)( \p_1 - \iu \p_2) & 0
	\end{bmatrix}
	\begin{bmatrix}
	\tilde{u}_{1}
	\\
	\tilde{u}_{2}
	\end{bmatrix}
	= \frac{\omega-\omega^*}{\sqrt\delta}
	\begin{bmatrix}
	\tilde{u}_{1}
	\\
	\tilde{u}_{2}
	\end{bmatrix},
	$$ 
	where $c, S_1$ and $S_2$ are now defined using $\alpha^*=\alpha_2^*$.
\end{remark}

\section{Numerical illustrations} \label{sec:num} 
In this section, we illustrate the main result of this paper, namely Theorem \ref{thm:main}, in the case of circular bubbles. We do this by numerically computing the eigenmodes close to the Dirac points for the honeycomb lattice. For comparison, in Section \ref{sec:num_square}, we also compute the eigenmodes in the case of a square lattice of bubbles. This will show the necessity of having a honeycomb lattice, instead of the simpler square lattice, in order to achieve near-zero index in the sub-wavelength regime. This also serves to illustrate and numerically verify the conclusions from \cite{homogenization}.

The eigenmodes are computed by discretising the operator $\Acal_\delta^{\alpha,\omega}$ from equation \eqnref{eq:A} using the multipole method as described in \cite{honeycomb,AFLYZ}. All computations are made for circular bubbles with radius $R= 0.2$. Moreover, the material parameters are  $\rho = \kappa = 1000$, $\rho_b = \kappa_b = 1$, which gives $\delta = 10^{-3}$, $v=1$, and $v_b=1$.

For simplicity, we will perform the computations with the scaling $s=1$. Observe that Theorem \ref{thm:main} is valid for small $s$, which will instead be achieved using the rescaling properties of the eigenmodes. From the proof of Lemma \ref{lem:Bloch_scale} we have $u_{\pm,s}^{\alpha^*/s+\tilde{\alpha}}(x) = u_\pm^{\alpha^* + s \tilde{\alpha}}(x/s)$. Therefore, a small scaling factor $s$ corresponds to choosing quasi-periodicities $\alpha$ sufficiently close to $\alpha^*$ and choosing a large range of $x$. This justifies the choice $s=1$, and allows us to study the limiting behaviour without changing the geometry of the differential equation.

\subsection{Honeycomb lattice} \label{sec:num_honeycomb}
Recall from \eqnref{eq:u} that the eigenmodes can be expressed as
\begin{equation} \label{eq:uagain}
u(\Bx) = \begin{cases}
\Scal_D^{\alpha, k_b} [\phi](\Bx),\quad \Bx\in D, \\
\Scal_D^{ \alpha,k}[\psi](\Bx),\quad \Bx\in Y\setminus\overline{D},
\end{cases}
\end{equation}
where $\mathcal{A}_\delta^{\alpha,\omega} \left(\begin{smallmatrix} \phi \\ \psi \end{smallmatrix}\right)=0.$ Then $(\phi,\psi)$ can be numerically computed as an eigenvector of the discretised operator $\Acal_\delta^{\alpha,\omega}$ corresponding to the eigenvalue $0$. Moreover, $u(x)$ can be computed by \eqnref{eq:uagain}, extended quasi-periodically to the whole space. We will consider the eigenmodes at frequencies $\omega$ with $\omega-\omega^* = \beta\sqrt{\delta}$.

First, we consider the small-scale behaviour of the eigenmodes. The small scale corresponds to $e^{\iu\tilde\alpha\cdot x} \approx 1$, so Theorem \ref{thm:main} shows that the eigenfunctions are given by 
$$u_{\pm}^{\alpha^*+\tilde{\alpha}}(x) = A  S_{1} (x) + B S_{2} (x) + \O(s).$$
Equation (\ref{abpm}) shows that $A = A_\pm = \pm\frac{1}{\sqrt{2}}e^{\iu(\theta_c+\theta)}$ and $B = B_\pm = \frac{1}{\sqrt{2}}$, where $\theta_c$ and $\theta$ are the arguments of $c$ and $\tilde{\alpha}$, respectively, and the sign coincides with the sign of $\beta$. To pick a unique eigenmode, we choose $\theta=0$, \ie{}, the quasi-periodicity $\tilde\alpha = \left( \begin{smallmatrix} \tilde\alpha_1 \\ 0 \end{smallmatrix}\right)$.

Figure \ref{fig:small-scale1} shows the function $$u^*(x) = A_-  S_{1} (x) + B_- S_{2} 
(x),$$ which is the first Dirac eigenmode in the limit $\beta \rightarrow 0^-$. It can be seen that the eigenmode is highly oscillating and oscillates between $-1$ and $1$ within one hexagon of bubbles. Moreover, this eigenmode has no large-scale oscillations. The second eigenmode, corresponding to $\beta \rightarrow 0^+$, has the same qualitative features.

Next, we consider the large-scale behaviour of the eigenmodes. Figure \ref{fig:efunc_h1} shows the real part of the first eigenmode $u_{-}^{\alpha^*+\tilde{\alpha}}$ for $\beta = 8\cdot10^{-3}$. It can be seen that the eigenmode is oscillating with a low frequency in the large scale. Moreover, it is clear that the eigenmode consists of two superimposed fields, corresponding to the parts $Ae^{\iu\tilde\alpha\cdot x}  S_{1} (x)$ and $B e^{\iu\tilde\alpha\cdot x}S_{2} (x)$. These fields are phase-shifted, due to the factor $e^{\iu(\theta_c + \theta)}$ in $A$.

To demonstrate the near-zero effective refractive index of the bubbly honeycomb crystal, we consider the large-scale spatial oscillation frequency close to the Dirac points. From (\ref{eq:hom}), we know that for frequencies $\omega = \omega^*+ \epsilon$ close to the Dirac frequency, each component $\tilde u_j, j=1,2,$ of the macroscopic field oscillates at a spatial frequency 
$$f = \frac{|\epsilon|}{|c|\lambda_0\sqrt{\delta}}.$$
To verify this relation, we compute the large-scale spatial frequency of the eigenmodes for $\epsilon$ in the range $\epsilon \in [-0.01,0.01]$, as shown in Figure \ref{fig:linear}. It can be seen that the relation is linear for $\epsilon$ close to $0$. This verifies \eqnref{eq:hom}, and shows that close to the Dirac frequency the macroscopic behaviour of the honeycomb crystal can be described as a near-zero refractive index material. However, we emphasize that the homogenization approach is valid  \emph{only} in the large scale; this will fail to capture the small-scale oscillations. Also, we emphasize the counter-intuitive result that despite not being located at $\Gamma$, the eigenmodes close to the Dirac point show close to zero phase change across the crystal. Indeed, this is a consequence of the near-zero property and the zero phase change of the microscopic field across a hexagon.

\begin{figure}[p]
	\centering
	\vspace{-40pt}
	\begin{subfigure}[b]{0.48\linewidth}
		\includegraphics[scale=0.42]{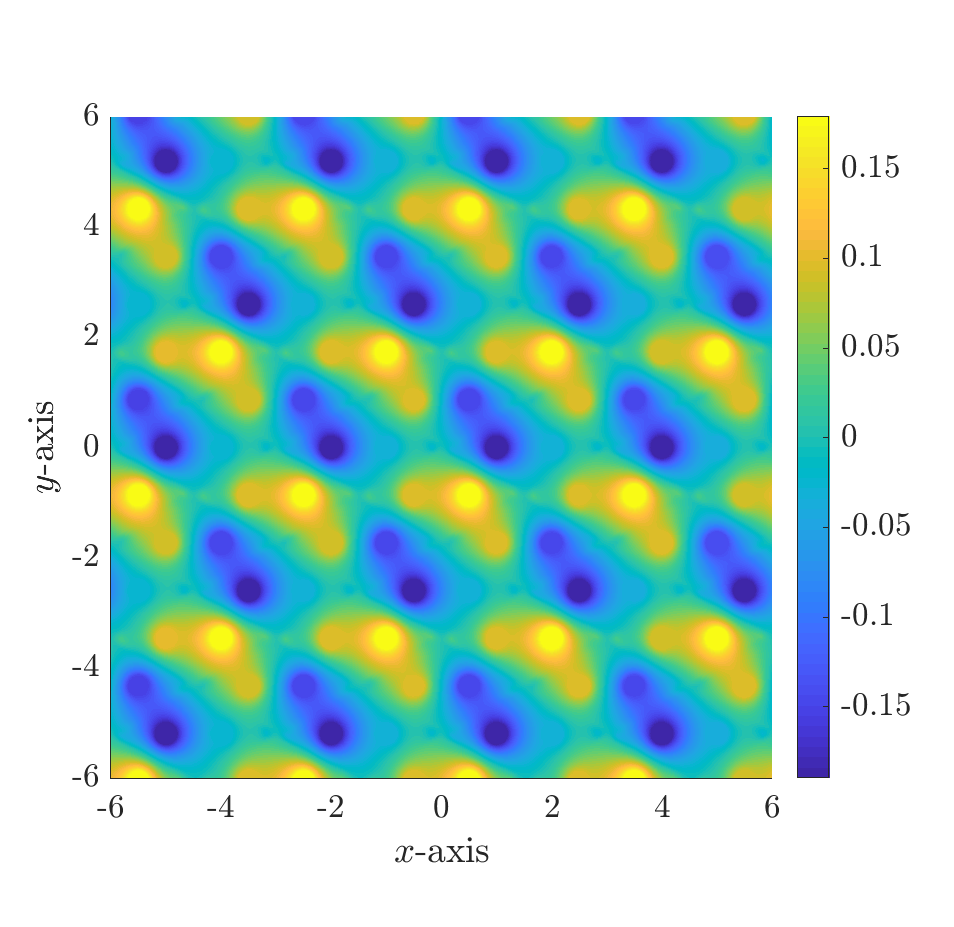}
		\vspace{-10pt}
		\caption{Real part of $u^*$.}
		\label{fig:Smr}
	\end{subfigure}
	\hspace{3pt}
	\begin{subfigure}[b]{0.48\linewidth}
		\includegraphics[scale=0.42]{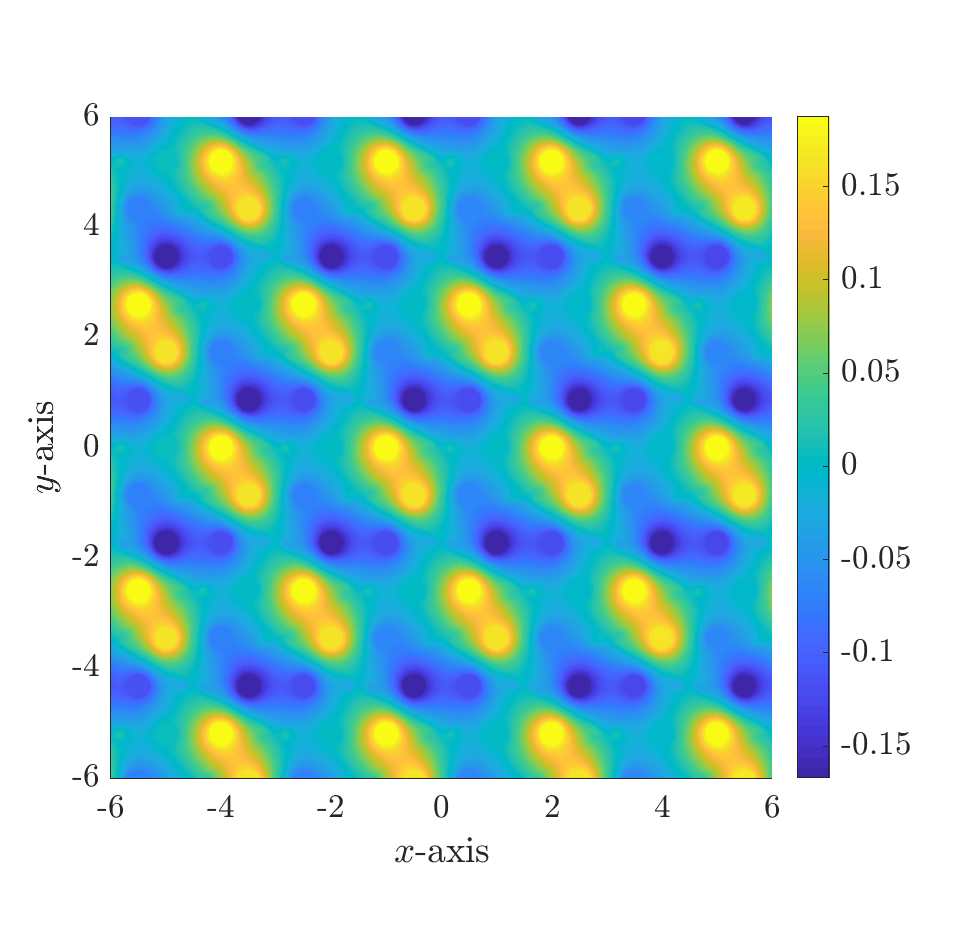}
		\vspace{-10pt}
		\caption{Imaginary part of $u^*$.}
		\label{fig:Smi}
	\end{subfigure}
	\caption{Small-scale behaviour of the first Bloch eigenfunction $u^*$.} \label{fig:small-scale1}
\end{figure}

\begin{figure}[p]
	\centering
	\vspace{-10pt}
	\begin{subfigure}[b]{0.43\linewidth}
		\hspace{-10pt}
		\includegraphics[scale=0.37]{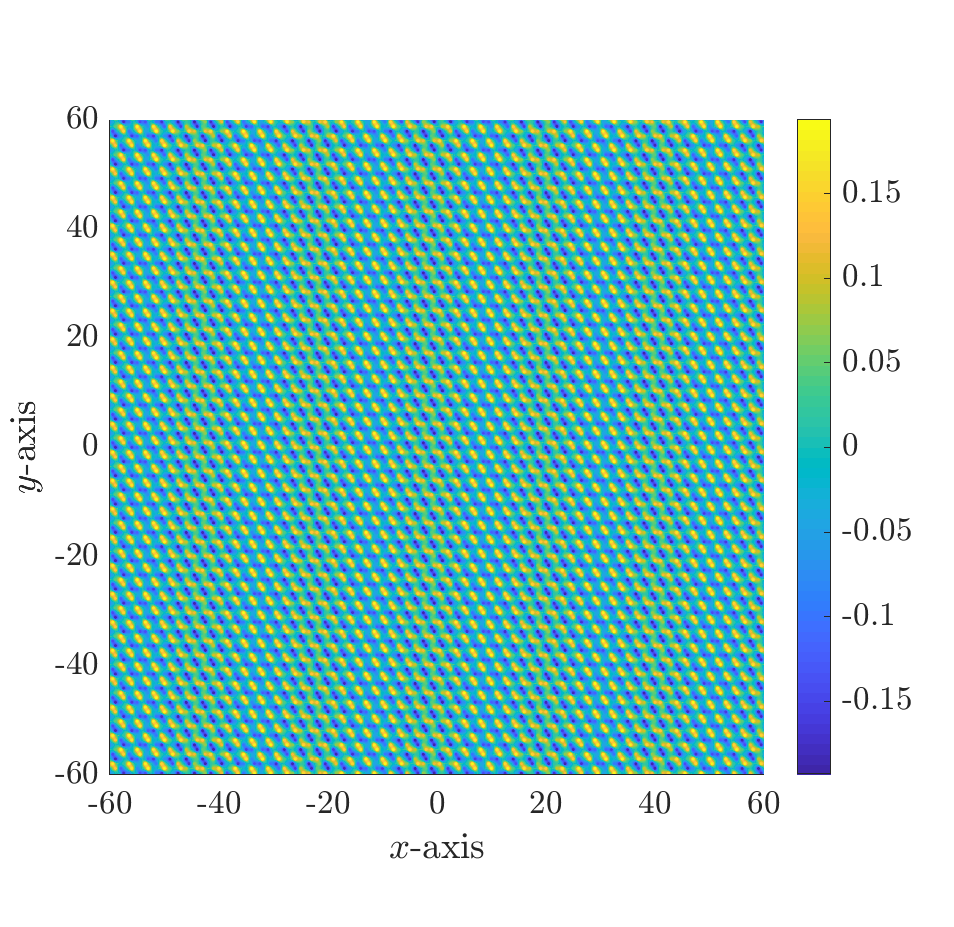}
		\vspace{-15pt}
		\caption{Two-dimensional plot.}
		\label{fig:h2D1}
	\end{subfigure}
	\hspace{3pt}
	\begin{subfigure}[b]{0.52\linewidth}
		\includegraphics[scale=0.43]{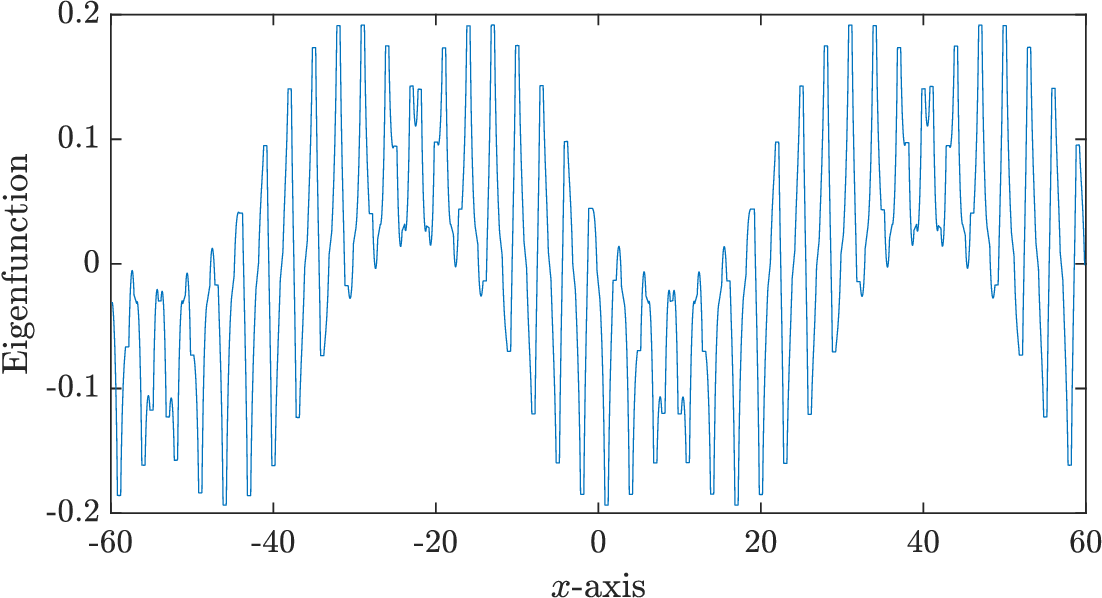}
		\caption{One-dimensional plot along the $x$-axis.}
		\label{fig:h1D1}
	\end{subfigure}
	\caption{Real part of first Bloch eigenfunction of the honeycomb lattice shown over many unit cells.} \label{fig:efunc_h1}
\end{figure}

\begin{figure}[p]
	\centering
	\includegraphics[scale=0.45]{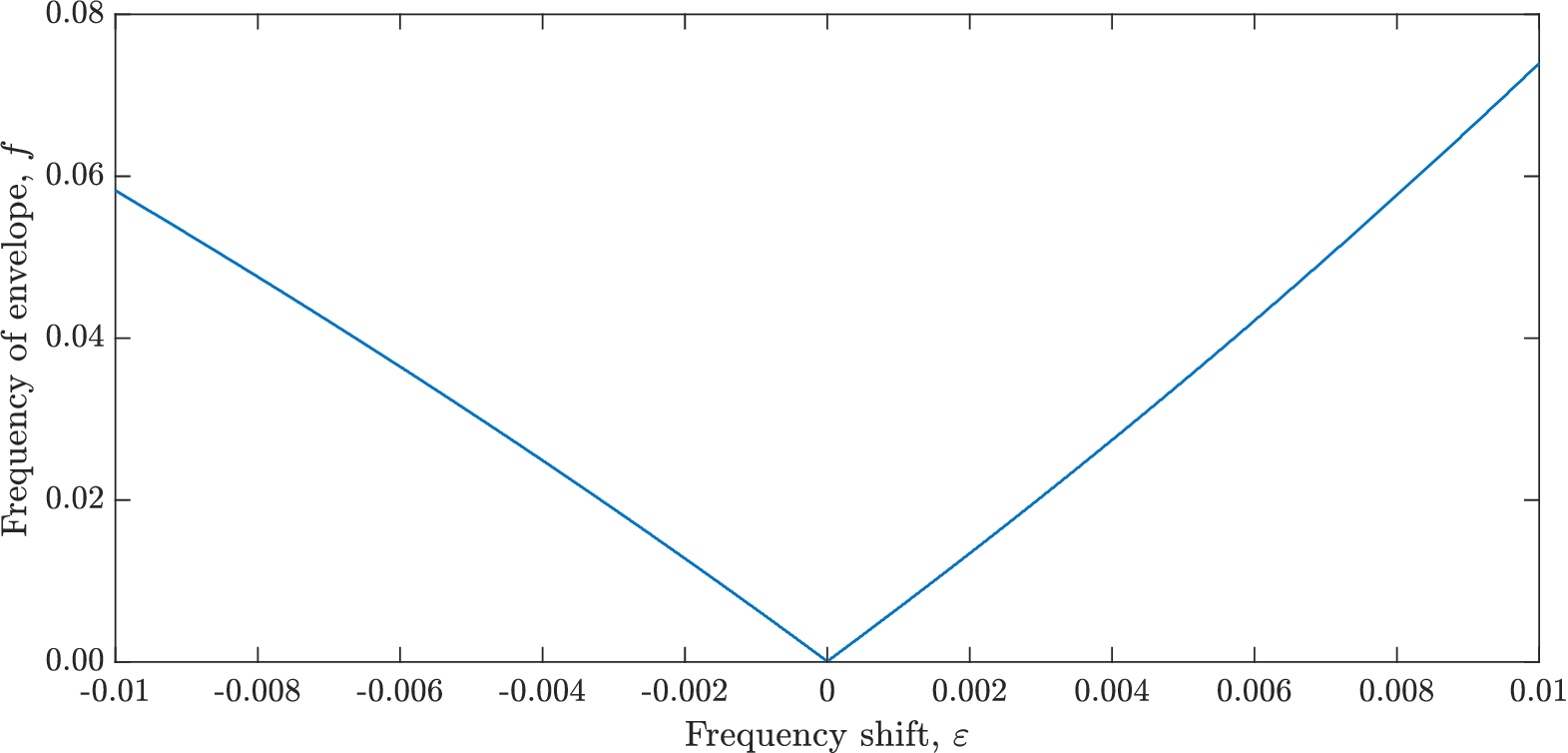}
	\caption{Spatial frequency of the macroscopic functions $\tilde u_j$ for different frequency shifts $\epsilon:=\beta\sqrt{\delta}$.}
	\label{fig:linear}
\end{figure}

\subsection{Square lattice} \label{sec:num_square}
In this section, we perform the same numerical experiments as in Section \ref{sec:num_honeycomb} but for the case of a square lattice of bubbles. We begin by recalling the main result from \cite{homogenization}. Consider now a square lattice with unit cell $Y$ as depicted in Figure \ref{fig:square_lattice}. The corresponding dispersion relation has a bandgap between the first and second bands, and the critical frequency $\omega^*$ of the first band is achieved at the symmetry point $\alpha^* = M = (\pi,\pi)$ in the Brillouin zone.

\begin{figure}[h]
	\centering
	\begin{tikzpicture}
	\begin{scope}[xshift=-4cm,scale=1]
	\coordinate (a) at (1,0);		
	\coordinate (b) at (0,1);	
	
	\draw (-0.5,-0.5) -- (0.5,-0.5) -- (0.5,0.5) -- (-0.5,0.5) -- cycle; 
	\draw (0,0) circle(6pt);
	\draw[opacity=0.2] (-0.5,0) -- (0,0)
	(0.5,0) -- (0,0)
	(0,-0.5) -- (0,0)
	(0,0.5) -- (0,0);
	
	\begin{scope}[shift = (a)]
	\draw (0,0) circle(6pt);
	\draw[opacity=0.2] (-0.5,0) -- (0,0)
	(0.5,0) -- (0,0)
	(0,-0.5) -- (0,0)
	(0,0.5) -- (0,0);
	\end{scope}
	\begin{scope}[shift = (b)]
	\draw (0,0) circle(6pt);
	\draw[opacity=0.2] (-0.5,0) -- (0,0)
	(0.5,0) -- (0,0)
	(0,-0.5) -- (0,0)
	(0,0.5) -- (0,0);
	\end{scope}
	\begin{scope}[shift = ($-1*(a)$)]
	\draw (0,0) circle(6pt);
	\draw[opacity=0.2] (-0.5,0) -- (0,0)
	(0.5,0) -- (0,0)
	(0,-0.5) -- (0,0)
	(0,0.5) -- (0,0);
	\end{scope}
	\begin{scope}[shift = ($-1*(b)$)]
	\draw (0,0) circle(6pt);
	\draw[opacity=0.2] (-0.5,0) -- (0,0)
	(0.5,0) -- (0,0)
	(0,-0.5) -- (0,0)
	(0,0.5) -- (0,0);
	\end{scope}
	\begin{scope}[shift = ($(a)+(b)$)]
	\draw (0,0) circle(6pt);
	\draw[opacity=0.2] (-0.5,0) -- (0,0)
	(0.5,0) -- (0,0)
	(0,-0.5) -- (0,0)
	(0,0.5) -- (0,0);
	\end{scope}
	\begin{scope}[shift = ($-1*(a)-(b)$)]
	\draw (0,0) circle(6pt);
	\draw[opacity=0.2] (-0.5,0) -- (0,0)
	(0.5,0) -- (0,0)
	(0,-0.5) -- (0,0)
	(0,0.5) -- (0,0);
	\end{scope}
	\begin{scope}[shift = ($(a)-(b)$)]
	\draw (0,0) circle(6pt);
	\draw[opacity=0.2] (-0.5,0) -- (0,0)
	(0.5,0) -- (0,0)
	(0,-0.5) -- (0,0)
	(0,0.5) -- (0,0);
	\end{scope}
	\begin{scope}[shift = ($-1*(a)+(b)$)]
	\draw (0,0) circle(6pt);
	\draw[opacity=0.2] (-0.5,0) -- (0,0)
	(0.5,0) -- (0,0)
	(0,-0.5) -- (0,0)
	(0,0.5) -- (0,0);
	\end{scope}
	\end{scope}
	
	\draw[dashed,opacity=0.5,->] (-3.9,0.65) .. controls(-2.9,1.8) .. (0.5,0.7);
	\begin{scope}[xshift=2cm,scale=2.8]	
	\coordinate (a) at (1,{1/sqrt(3)});		
	\coordinate (b) at (1,{-1/sqrt(3)});	
	\coordinate (Y) at (1.8,0.45);
	\coordinate (c) at (2,0);
	\coordinate (x1) at ({2/3},0);
	\coordinate (x0) at (1,0);
	\coordinate (x2) at ({4/3},0);

	\pgfmathsetmacro{\rb}{0.25pt}
	\pgfmathsetmacro{\rs}{0.2pt}\
	
	\draw[->] (-0.5,-0.5) -- (-0.5,0.5) node[left]{$l_2$}; 
	\draw[->] (-0.5,-0.5) -- (0.5,-0.5) node[below]{$l_1$}; 
	\draw (0.5,-0.5) -- (0.5,0.5) -- (-0.5,0.5);
	\draw (0,0) circle(6pt);
	\draw (0.3,0) node{$D$};
	
	\draw (0.5,0.5) node[right]{$Y$};
	\end{scope}
	\end{tikzpicture}
	\caption{Illustration of the square lattice crystal and quantities in the fundamental domain $Y$.} \label{fig:square_lattice}
\end{figure}
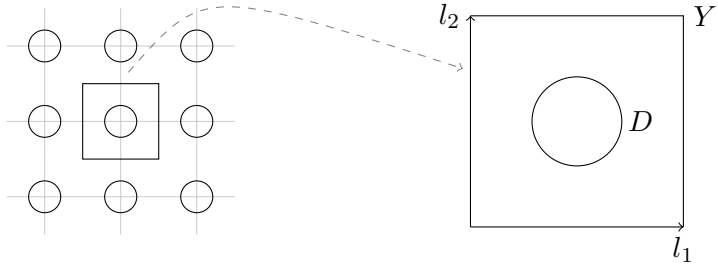

Let now $S_D^{\alpha,k}$ denote the single-layer potential for the square lattice, defined analogously as in Section \ref{sec:setup_G} but with $\Lambda$ and $\Lambda^*$ being the square lattice and dual square lattice, respectively. Define the function 
$$S(x) = \Scal^{\alpha,0}_D\left[\psi^{\alpha^*}\right](x), \quad x\in \R^2,$$ 
where $\psi^{\alpha} = \left(\Scal^{\alpha,0}_D\right)^{-1}\left[\chi_{\partial D}\right]$.

We now consider the rescaled crystal with unit cell $sY, \ s>0$. Again, we assume that the order of $\omega^*$ is fixed, \ie{},  $\delta = \mu s^2,$ for some $\mu>0$.
Then the eigenmodes of the rescaled square crystal are given in the following theorem, which is the analogue of Theorem \ref{thm:main} for the case of a square lattice.
\begin{theorem}[\cite{homogenization}]
	For frequencies $\omega$ close to the critical frequency $\omega^*$, namely, $\left(\omega^*\right)^2 - \omega^2  = \O(s^2)$, the following asymptotic behaviour of the Bloch eigenfunction $u^{\alpha^*/s + \tilde{\alpha}}_s$ holds:
	$$
	u_{s}^{\alpha^*/s+\tilde{\alpha}}(x) = e^{\iu\tilde\alpha \cdot x} S\left(\frac{x}{s}\right) + \O(s),
	$$  
	where the macroscopic field $\tilde{u}:= e^{\iu\tilde{\alpha}\cdot x}$ satisfies the Helmholtz equation
\beq \Delta \tilde{u} + \frac{(\omega^*)^2-\omega^2}{\delta\tilde\lambda^2} \tilde{u} = 0, \label{eq:hom_square} \eeq
	which can be considered as a homogenized equation for the square bubbly crystal, while the microscopic field $S$ varies on the scale of $s$.
\end{theorem}
The isotropic form of the macroscopic equation \eqnref{eq:hom_square} follows since $D$ is a circle, and an expression for $\tilde\lambda$ is given in \cite{homogenization}.

\begin{figure}[p]
	\centering
	\vspace{-40pt}
	\includegraphics[scale=0.42]{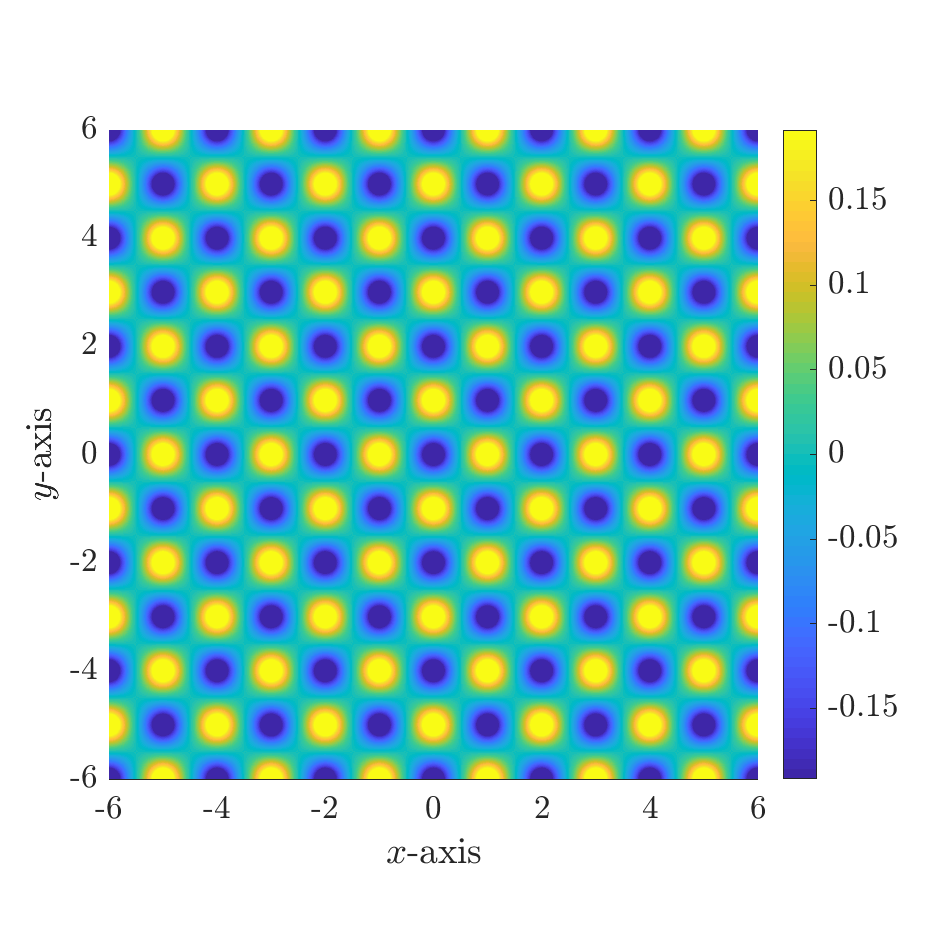}
	\vspace{-10pt}
	\label{fig:S}
	\caption{Real part of small-scale behaviour of the Bloch eigenfunction of the square lattice at $\alpha=\alpha^*$ (imaginary part close to 0).} \label{fig:small-scale_square}
\end{figure}

\begin{figure}[p]
	\centering
	\vspace{-10pt}
	\begin{subfigure}[b]{0.43\linewidth}
		\hspace{-10pt}
		\includegraphics[scale=0.37]{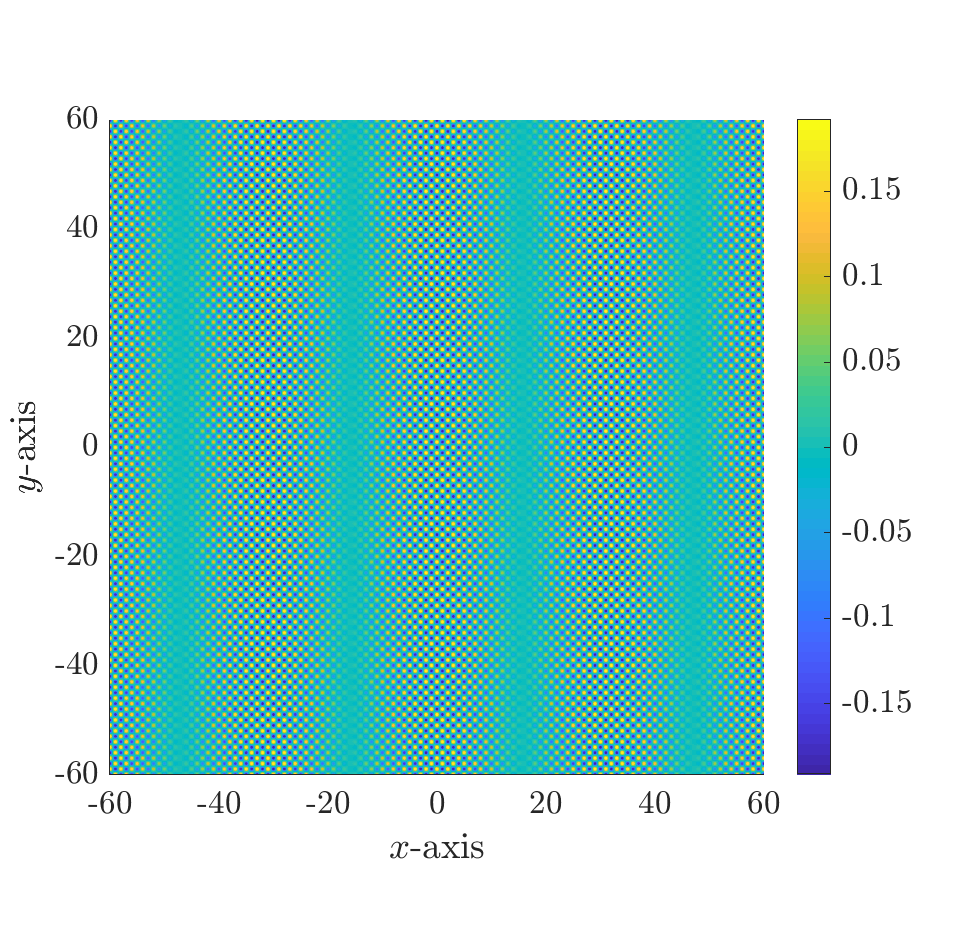}
		\vspace{-15pt}
		\caption{Two-dimensional plot.}
		\label{fig:s2D}
	\end{subfigure}
	\hspace{3pt}
	\begin{subfigure}[b]{0.52\linewidth}
		\includegraphics[scale=0.43]{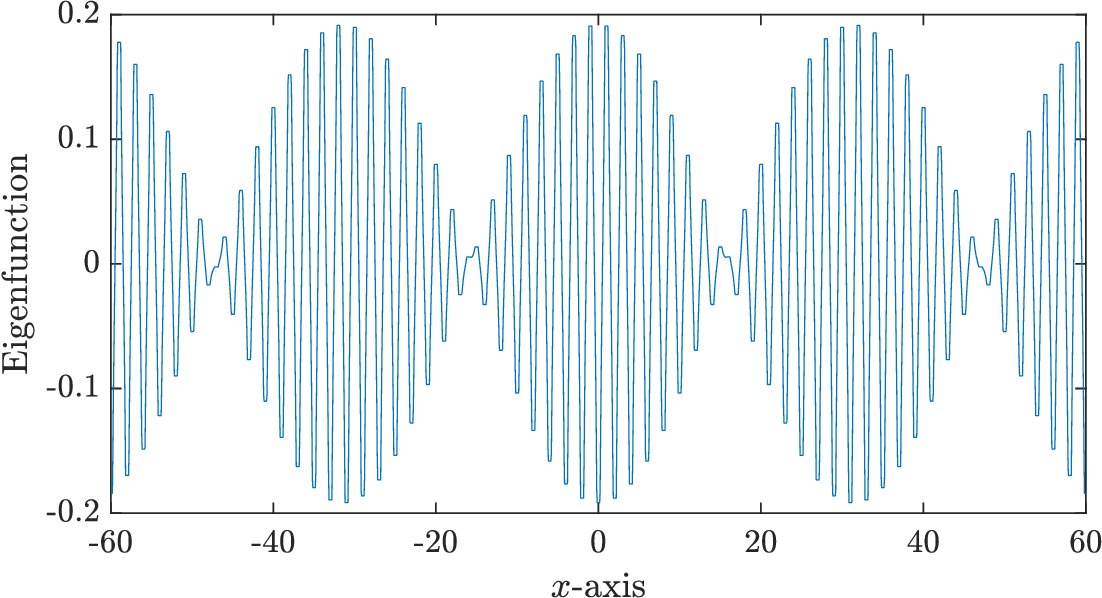}
		\caption{One-dimensional plot along the $x$-axis.}
		\label{fig:s1D}
	\end{subfigure}
	\caption{Real part of Bloch eigenfunction of the square lattice shown over many unit cells.} \label{fig:efunc_s}
\end{figure}

\begin{figure}[p]
	\centering
	\includegraphics[scale=0.45]{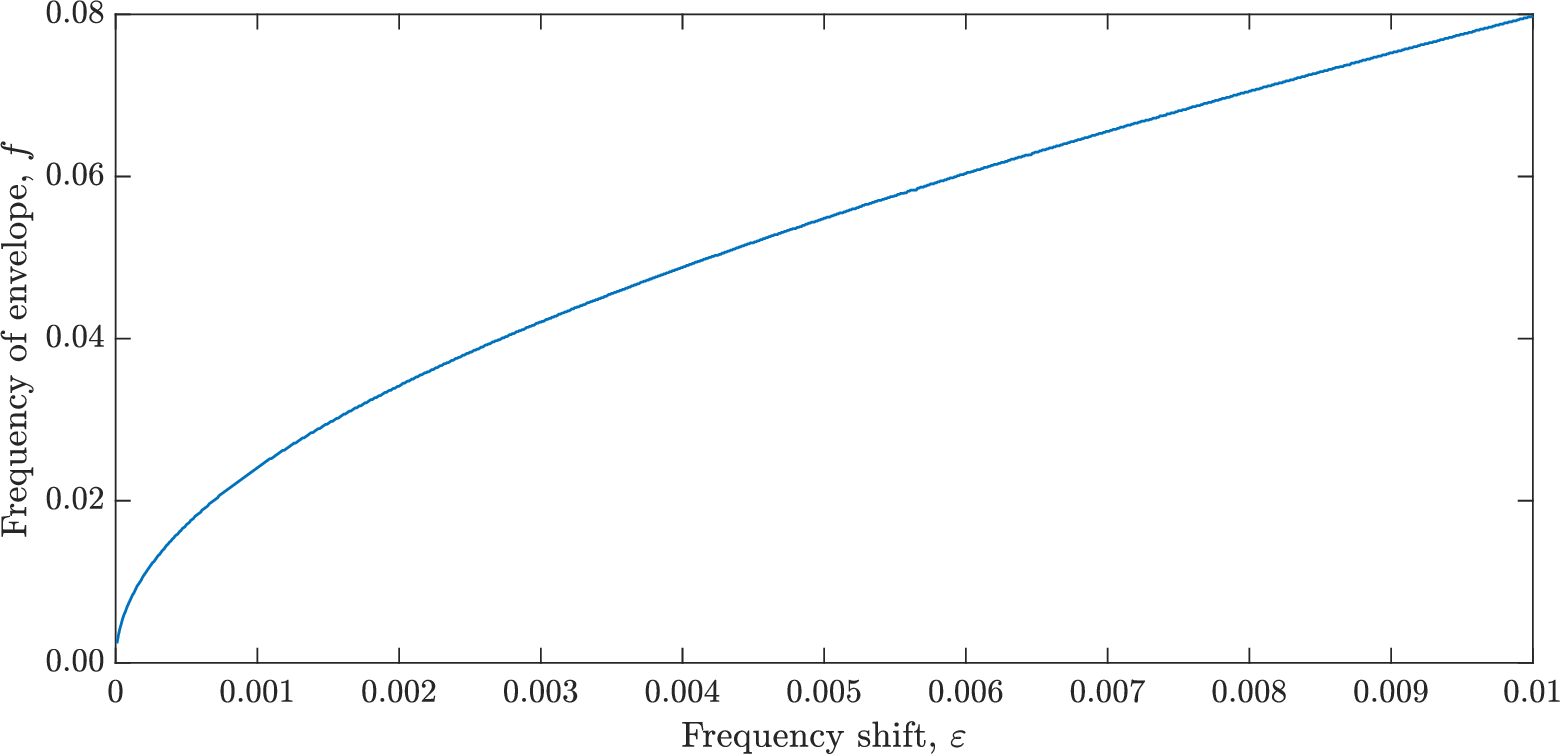}
	\caption{Spatial frequency of the envelope function $\tilde u$ for different frequency shifts $\epsilon$ in the case of a square lattice.}
	\label{fig:root}
\end{figure}

We now compute the eigenmodes of the square crystal close to the critical frequency, namely, $\omega=\omega^*-\epsilon$. The small-scale behaviour of the eigenmodes, \ie{}, the function $S\left(\frac{x}{s}\right)$, is shown in Figure \ref{fig:small-scale_square}. It can be seen that the function oscillates at the scale of the bubbles. Next, the large-scale behaviour is considered. Figure \ref{fig:efunc_s} shows the Bloch eigenfunction over many unit cells for $\tilde\alpha = \left(\begin{smallmatrix} \tilde{\alpha}_1 \\ 0 \end{smallmatrix}\right)$ and $\omega^*-\omega = 6\cdot10^{-5}$. Similarly as in the case of a honeycomb crystal, the eigenfunction varies at the large scale with a low-frequency macroscopic field $\tilde u$. 

To illustrate the macroscopic equation \eqnref{eq:hom_square}, the spatial frequency $f$ of the macroscopic field is computed for $\epsilon$ in the range $\epsilon \in [0,0.01]$. Observe that due to the bandgap above $\omega^*$, no eigenmodes exists for $\epsilon < 0$. Figure \ref{fig:root} shows the spatial frequency $f$ for different $\epsilon$. This figure shows that $f$ scales like $\sqrt{\epsilon}$ for small $\epsilon > 0$. This is consistent with  (\ref{eq:hom_square}), from which we expect a spatial frequency 
$$f = \sqrt{\frac{(\omega^*)^2-\omega^2}{\tilde\lambda^2\delta}} = \sqrt{\epsilon}\left(\sqrt{\frac{\omega^*+\omega}{\tilde\lambda^2\delta}}\right).$$

In summary, Figures \ref{fig:efunc_h1} and \ref{fig:efunc_s} show that both in the cases of a honeycomb lattice and a square lattice, the eigenmodes have a periodic small-scale oscillation and a large-scale macroscopic oscillation. However, Figures \ref{fig:linear} and \ref{fig:root} show that the spatial frequency of the macroscopic fields have different asymptotic behaviours close to the critical frequency, due to the fact that the square lattice cannot be mapped to a zero-index effective material.

\section{Concluding remarks} \label{sec:conclusion}
In this paper we have derived, for the first time, the equation governing wave propagation in a honeycomb crystal of sub-wavelength resonators near the Dirac points. We have decomposed the Bloch eigenfunctions as the sum of two eigenmodes. 
The effective equation for the envelope of each of these eigenmodes is of Helmholtz-type with near-zero refractive index. Furthermore, we have shown that the two envelopes are phase-shifted and satisfy a system of Dirac equations. A comparison with a square lattice crystal shows the great potential of using honeycomb crystals of sub-wavelength resonators as near-zero materials. In a forthcoming work, we plan to study topological phenomena in bubbly time-dependent crystals and derive their effective properties. These materials may exhibit a Dirac cone  near the point $\Gamma$ and therefore, may exhibit finite acoustic impedance and a high transmittance \cite{Dubois2017}. We also plan to mathematically analyse wave propagation phenomena in near-zero refractive index materials.

\end{document}